\documentclass{article}

\usepackage[left=1.5in,right=1.5in,top=1.5in,bottom=1.5in]{geometry}

\usepackage{amsfonts,amscd,amssymb,enumerate}
\usepackage{amsthm}
\usepackage{amsmath} 
\usepackage{mathrsfs} 
\usepackage{bbm} 
\usepackage{tikz}%bb font numbers
\usepackage[colorlinks]{hyperref}
\theoremstyle{plain}
\newtheorem{thm}{Theorem}[section]
\newtheorem{lemma}[thm]{Lemma}
\newtheorem{prop}[thm]{Proposition}
\newtheorem{cor}[thm]{Corollary}
\theoremstyle{definition}
\newtheorem*{defn}{Definition}

\theoremstyle{example}
\newtheorem{example}{Example}

\theoremstyle{remark}
\numberwithin{equation}{section}

\usepackage{etex}
\usepackage{graphicx}
\usepackage{subfig}

\usepackage[all]{xy}

\newcommand{\eqdef}{\overset{\mbox{\tiny def}}{=}}

\def\R{\mathbb R}
\def\Z{\mathbb Z}
\def\E{\mathbb{E}}

\def\ZZ{\mathbb{Z}}

\def\cD{\mathcal{D}}

\def\cT{\mathcal{T}}

\def\RR{\mathbb{R}}

\title{An asymptotic relationship between coupling methods for stochastically modeled population processes}

\author{David F. Anderson$^{1}$ and Masanori Koyama$^{2}$}

\begin{document}

\maketitle

\begin{abstract}
This paper is concerned with elucidating a relationship between two common coupling methods for the continuous time Markov chain models utilized in the cell biology literature.  The couplings considered here are primarily used in a computational framework by providing reductions in variance for different Monte Carlo estimators, thereby allowing for significantly more accurate results for a fixed amount of computational time.  Common applications of the couplings include the estimation of parametric sensitivities via finite difference methods and the estimation of expectations via multi-level Monte Carlo algorithms.  While a number of coupling strategies have been proposed for the models considered here, and a number of articles have experimentally compared the different strategies, to date there has been no mathematical analysis describing the connections between them.   Such analyses are critical in order to determine the best use for each.  In the current paper, we show a connection between the common reaction path (CRP) method and the split coupling (SC) method, which is termed coupled finite differences (CFD) in the parametric sensitivities literature.  In particular, we show that the two couplings are both limits of a third coupling strategy we call the ``local-CRP'' coupling, with the split coupling method arising as a key parameter goes to infinity, and the common reaction path coupling arising as the same parameter goes to zero.  The analysis helps explain why the split coupling method  often provides a lower variance than does the common reaction path method, a fact previously  shown  experimentally.

%Motivated by  computational methods for the stochastic models commonly found in the chemistry and biochemistry literature, where variance reduction methods are critical to achieve This paper is concerned with couplings methods We present an asymptotic relationship between two coupling methods for
%Poisson random time change representations of discrete stochastic
%processes. Different ways of coupling arose as techniques in finite
%different method sensitivity analysis. In \cite{AndersonCFD},
%Anderson invented an efficient finite difference method using the
%coupling of Kurtz \cite{Kurtz78}.  In \cite{AndersonMLMC}, he also
%showed that the Multi level Monte Carlo method of \cite{Giles} proves to be
%an epic simulation method for discrete Markov processes only with the coupling of
%Kurtz. However, the exact mathematical reason behind the superiority of
%Kurtz' coupling is unknown.  In this article we show a precise
%mathematical connection between the coupling of Kurtz and the Common
%Reaction Path coupling (CRP), the most intuitive and most commonly used
%coupling when using the time change represenation.
\end{abstract}

\footnotetext[1]{Department of Mathematics, University of
  Wisconsin, Madison, Wi. 53706, anderson@math.wisc.edu.}

\footnotetext[2]{Department of Systems Science, Graduate School of Informatics, Kyoto University, Integrated Systems Biology Lab, koyama.masanori@gmail.com. Corresponding Author.}


\section{Introduction}
\label{sec:intro}
Models of biochemical reaction networks with stochastic dynamics have become increasingly popular in the science literature over the previous fifteen years where they are often studied via computational methods and, in particular, Monte Carlo methods.  These computational methods tend to be extremely expensive and time-consuming without the use of variance reduction techniques.  One of the most common ways to achieve a large reduction of variance is to couple two relevant processes in order to increase their covariance.  There are three main couplings found in the relevant literature: (i)  the use of common random numbers (CRN), (ii) the common reaction path (CRP) coupling \cite{Khammash2010}, and (iii) a \textit{split coupling} (SC) method  termed  coupled finite differences in the setting of parametric sensitivities \cite{AndCFD2012,AndHigham2012}.  It has been observed in the literature that both the CRP and SC couplings are far superior to the CRN coupling in terms of variance reduction \cite{AndCFD2012,Khammash2010,Srivastava2013}.  It has also been observed through examples that the SC method tends to perform much better than the CRP method, though some exceptions  exist \cite{AndCFD2012,Srivastava2013}.  To the best of the authors' knowledge there has to date been no analytical work on understanding the connections between these two couplings.  In the present paper we prove that both the CRP and SC couplings arise naturally as different limits of a third family of couplings we term the \textit{local-CRP} coupling.  In particular, the CRP coupling arises as a limit in which the local-CRP coupling is as loosely coupled as possible, whereas the SC coupling arises from a limit of the local-CRP ``recoupling'' as often as possible.  Such an analysis sheds light on why the split coupling often provides a lower variance than does the CRP coupling.

The outline for the remainder of the paper is as follows.  In Section \ref{sec:MathModel}, we formally present the mathematical models considered in this paper, together with a brief description of the  computational methods that serve as motivation for the present work.  In Section \ref{sec:couplings}, we present the different coupling strategies for the models presented in Section \ref{sec:MathModel}.   In Section \ref{sec:Analysis}, we state and prove our main results.  In Section \ref{sec:examples}, we provide numerical examples demonstrating our main results, and in Section \ref{sec:discussion} we conclude with some brief remarks.

\section{Mathematical model and motivating computational methods}
\label{sec:MathModel}

Motivated by models in biochemistry, we consider continuous time
Markov chain models in $\Z^d$, in which the $i$th component of the
process typically represents the number of molecules of ``species''
$i$ present in the system.  The transitions of the chain are specified by vectors, $\zeta_k
\in \Z^d$, for $k \in \{1,\dots,R\}$ with $R < \infty$, determining
the net change in the chain due to the occurrence of a single
``reaction,'' and by the intensity functions $\lambda_k: \Z^d \to
\R_{\ge 0}$, which determine the rate at which the different reactions
are occurring.\footnote{Intensity functions are termed ``propensity''
  functions in the biochemistry literature.}  Specifically, letting $N_k(t)$ be the number of times transition $k \in \{1,\dots,R\}$ has occurred by time $t\ge 0$, we will consider the continuous time Markov chain $X$ satisfying the equation
\begin{equation*}
	X(t) = X(0) + \sum_{k = 1}^R N_k(t) \zeta_k,
\end{equation*}
where $N_k$ is a counting process with local intensity function $\lambda_k$.   That is, $\{N_k\}$ are the counting processes for which the processes
\[
	N_k(t) - \int_0^t \lambda_k(X(s))ds
\]
are local martingales.  
One useful representation for the counting processes $N_k(t)$ is via time-changed unit-rate Poisson processes \cite{Anderson2007a,AndKurtz2011,Kurtz86,KurtzPop81},
\[
	N_k(t) = Y_k\left( \int_0^t \lambda_k (X(s)) ds\right),
\]
yielding the stochastic equation
\begin{equation}\label{eq:RTC_X}
	X(t) = X(0) + \sum_{k = 1}^R Y_k\left( \int_0^t \lambda_k (X(s)) ds\right)\zeta_k
\end{equation}
where $\{Y_k\}_{k = 1}^R$ is a collection of independent unit-rate Poisson processes.
Note that $X$ can also be specified by its infinitesimal generator, 
\begin{equation}\label{eq:gen_A}
	(Af)(x) = \sum_{k = 1}^R \lambda_k(x) (f(x+\zeta_k) - f(x)),
\end{equation}
where $f$ is any bounded function with compact support.

We denote $Z$ as the process on $\Z^d$ with the same transition directions $\{\zeta_k\}$ as $X$, but with intensities $\widetilde \lambda_k:\Z^d \to \R_{\ge 0}$.  That is, $Z$ is the Markov process with infinitesimal generator
\begin{equation}\label{eq:gen_B}
	(Bf) (x) = \sum_{k=1}^R \widetilde \lambda_k(x)(f(x+\zeta_k) - f(x)),
\end{equation}
and which satisfies the stochastic equation
\begin{equation}\label{eq:RTC_Z}
	Z(t) = Z(0) + \sum_{k = 1}^R Y_k\left( \int_0^t \widetilde \lambda_k (Z(s)) ds\right)\zeta_k,
\end{equation}
where $\{Y_k\}_{k = 1}^R$ is a collection of independent unit-rate Poisson processes.  In the remainder of the paper, we consider different ways to couple $X$ and $Z$ and provide an asymptotic relationship between two of the couplings.

%  
%  Specifically, we
%consider a dynamics of a population given by 
%\begin{equation*}
%	X(t) = X(0) + \sum_{k = 1}^R N_k(t) \zeta_k.
%\end{equation*}
%where each $N_k$ the number of times transition $k \in \{1,..., R\}$
%has occured by time $t \geq 0$.  
%One useful representation for the counting processes $N_k(t)$ are the
%unique processes for which 
%$$N_k(t) -  \int_0^t \lambda_k (X(s)) ds$$
%are local martingales, or the
% time-changed unit-rate Poisson processes \cite{Kurtz86,KurtzPop81},
%\[
%	N_k(t) = Y_k\left( \int_0^t \lambda_k (X(s)) ds\right),
%\]
%yielding the representation
%\begin{equation}\label{eq:RTC_X}
%	X(t) = X(0) + \sum_{k = 1}^R Y_k\left( \int_0^t \lambda_k (X(s)) ds\right)\zeta_k
%\end{equation}
%where $\{Y_k\}_{k = 1}^R$ is a collection of independent unit-rate Poisson processes.
%
%We denote $Z$ as the process on $\Z^d$ with the same transition
%directions $\{\zeta_k\}$ as $X$, but with counting processes $\widetilde
%N_k$ with intensities $\widetilde
%\lambda_k:\Z^d \to \R_{\ge 0}$.  That is, $Z$ is the process
%\begin{equation}\label{eq:RTC_Z}
%	Z(t) = Z(0) + \sum_{k = 1}^R Y_k\left( \int_0^t \widetilde \lambda_k (Z(s)) ds\right)\zeta_k,
%\end{equation}
%where $\{Y_k\}_{k = 1}^R$ is a collection of independent unit-rate Poisson processes.  In the remainder of the paper, we consider different ways to couple $X$ and $Z$ and provide an asymptotic relationship between two relevant couplings.

\subsection{Motivating computational methods}

We briefly present two computational methods that serve as the motivation for the analysis of the different coupling strategies: finite difference methods for parametric sensitivity analysis and multi-level Monte Carlo for the estimation of expectations.

\subsubsection{Parametric sensitivity analysis}

Suppose that  $\{X^\theta\} $ is a  parametric family of processes about $\theta$ on
a state space $E$, and $f:E \to \RR$ is some statistic of interest.  For example, $f(X(t)) = X_i(t)$ may provide the abundance of species $i$ at time $t \ge 0$.  It is common to wish to evaluate 
\begin{align}
\begin{split}
\frac{d}{d\theta} \E [f(X^\theta(t))] \approx \frac{\E[f(X^{\theta+ h}(t))] -  \E[f(X^{\theta}(t))]}{h}
\end{split} \label{eq:motivation_Derivs} 
\end{align}
 as a measurement of the  sensitivity of $\E [f(X^\theta(t))]$ with respect to $\theta$.  Such a strategy is usually called a \textit{finite difference} method.
We would  like to empirically evaluate the right-hand side of \eqref{eq:motivation_Derivs} in as efficient a manner as possible.  By coupling the processes $(X^{\theta + h}, X^{\theta})$, we may evaluate 
\[
h^{-1}\E[f(X^{\theta+ h}(t)) -  f(X^{\theta}(t))],
\]
with the magnitude of $\textsf{Var}(f(X^{\theta+ h}(t)) -  f(X^{\theta}(t)))$ determining the quality of the coupling.  In particular, we wish to minimize $\textsf{Var}(f(X^{\theta+ h}(t)) -  f(X^{\theta}(t)))$ without greatly increasing the computational cost of producing realizations of the coupled processes $(X^{\theta + h}, X^{\theta})$.
We explicitly note that in the setting of the previous section, we have 
$$\lambda_k(\cdot) = \eta_k(\theta, \cdot) , ~~~~ \widetilde \lambda_k(\cdot) =
\eta_k(\theta + h, \cdot),$$
where for each $k$,  $\{\eta_k(\theta, \cdot) : \RR^d \to \RR_{\ge 0}\}$  is a
parametric family of functions about $\theta$.  In this case, we have 
\[
	X = X^{\theta}, \quad \text{and} \quad  Z = X^{\theta+h}.
\]
As mentioned in Section \ref{sec:intro}, there has been a large amount of work in the literature on developing good coupling strategies for the estimation of parametric sensitivities via finite differences \eqref{eq:motivation_Derivs}; see, for example, \cite{AndCFD2012,AndSkubak,Khammash2010,Srivastava2013}.  To the best of the authors' knowledge there has been no mathematical analysis detailing the connections between the different couplings used, though see the discussion in Section \ref{sec:discussion} for details pertaining to a recent work by Arampatzis and Katsoulakis \cite{Markos}.

\subsection{Multi-level Monte Carlo} 

In \cite{Giles2008}, Mike Giles introduced the multi-level Monte Carlo (MLMC) method for the approximation of expectations of diffusion processes.  Specifically, if $X$ is the diffusion process of interest and $\{Z_\ell\}$ are a family of  approximations to $X$, with higher values of $\ell$ corresponding to better approximations, then we observe that for any function $f$ of interest,
\begin{align*}
	\E [f(X(t))] &\approx \E [f(Z_L) ]= \sum_{\ell = 1}^L \E[ f(Z_{\ell}(t)) - f(Z_{\ell-1}(t))] + \E [ f(Z_0)],
\end{align*}
where $L$ is chosen large enough so that $\left| \E[ f(X(t))] - \E [f(Z_L(t))]\right|$ is below some target accuracy.  It is typical to choose $Z_\ell$ to be the process produced by Euler-Maruyama with a step size of $M^{-\ell}$ for some $M \in \{2,3,\dots,7\}$.  If each term $f(Z_{\ell}(t)) - f(Z_{\ell-1}(t))$ is tightly coupled, then the variance of each of the intermediate estimators will be low, thereby moving the computational cost to the lowest level, $\E[ f(Z_0)]$, which can be estimated quickly via Euler-Maruyama with large time-steps. 

In \cite{AndHigham2012}, Anderson and Higham extended the multi-level Monte Carlo method to the setting of this paper by utilizing the split coupling detailed in Section \ref{sec:couplings}.  They further noted that an unbiased estimator can be produced for jump models by coupling the exact process $X$ with the approximate process with the finest time discretizatoin
\begin{align*}
	\E [ f(X(t)) ] =\E[ f(X(t)) - f(Z_L(t))] + \sum_{\ell = 1}^L \E[ f(Z_{\ell}(t)) - f(Z_{\ell-1}(t))] + \E[ f(Z_0)],
\end{align*}
where, again, it is the quality of the coupling at each level that determines the overall quality of the method.  

We point out that in the diffusive case the most natural coupling is to re-use the driving Brownian path for each of the coupled processes.  This is relatively easy to do via the Brownian bridge.  However, as will be noted in the next section, there are multiple natural couplings to choose from in the context of jump processes with state dependent intensity functions, and different choices lead to computational methods with vastly different computational complexities and, hence, runtimes.

\section{Different Couplings} 
\label{sec:couplings}

We return to the notation introduced at the beginning of Section
\ref{sec:MathModel} and focus our discussion on ways to couple $X$
and $Z$ with intensities $\lambda_k$ and $\widetilde \lambda_k$,
respectively. 

 \subsection{Split coupling}
We will begin by introducing the split coupling (SC), which first appeared as an analytic tool in \cite{Kurtz82} and later appeared in the context of computational methods in \cite{AndCFD2012,AndersonGangulyKurtz,AndHigham2012,AHS13,Markos}.  Let $a\wedge b \eqdef \min\{a,b\}$, and let
$\mathcal U$ and $\mathcal V$ be any c\`adl\`ag processes on $\RR^d$.  Then for each $k \in \{1,\dots,R\}$ we define the operators $r_{1k},r_{2k},$ and $r_{3k}$ via
\begin{align}
\begin{split}
r_{1k}(\lambda_k, \widetilde \lambda_k, \mathcal U, \mathcal V)(s) &\eqdef \lambda_k(\mathcal U(s)) \wedge \widetilde \lambda_k( \mathcal V(s))\\
r_{2k}(\lambda_k, \widetilde \lambda_k, \mathcal U, \mathcal V)(s) &\eqdef \lambda_k(\mathcal U(s)) -r_{1k}(\lambda_k, \widetilde \lambda_k, \mathcal U, \mathcal V)(s)\\
r_{3k}(\lambda_k, \widetilde \lambda_k, \mathcal U, \mathcal V)(s) &\eqdef \widetilde \lambda_k(\mathcal V(s)) - r_{1k}(\lambda_k, \widetilde \lambda_k, \mathcal U, \mathcal V)(s).
\end{split} \label{couple_rates}
\end{align}
The split coupling of the processes $X$ and $Z$ is then given by
\begin{align}
\begin{split}
X_{\text{sc}}(t) = X(0) + &\sum_{k=1}^R\Bigg\{ Y_{1k} \left(
 \int_0^t r_{1k}(\lambda_k, \widetilde \lambda_k, X_{\text{sc}}, Z_{\text{sc}})(s)  ds \right)  \\
&+   Y_{2k}
\left( \int_0^t r_{2k}(\lambda_k, \widetilde \lambda_k, X_{\text{sc}}, Z_{\text{sc}})(s)  ds \right)
\Bigg\}  \zeta_k \\
Z_{\text{sc}}(t) = Z(0) + &\sum_{k=1}^R \Bigg\{ Y_{1k} \left( \int_0^t
  r_{1k}(\lambda_k, \widetilde \lambda_k, X_{\text{sc}}, Z_{\text{sc}})(s)  ds
\right) \\
& +  Y_{3k} \left( \int_0^t   r_{3k}(\lambda_k, \widetilde \lambda_k, X_{\text{sc}}, Z_{\text{sc}})(s) ds  \right)  \Bigg\} \zeta_k,
\end{split} \label{eq:split_coupling}
\end{align}
where $\{Y_{1k}\}_{k = 1}^R\cup \{Y_{2k}\}_{k = 1}^R\cup \{Y_{3k}\}_{k = 1}^R$ are mutually independent unit-rate Poisson processes.  Note that $X_{\text{sc}}$ and $Z_{\text{sc}}$ share the family of counting processes determined by the Poisson processes $Y_{1k}$. Further note that $(X,Z)$ satisfying the stochastic equation \eqref{eq:split_coupling} is simply a continuous time Markov chain on $\Z^d\times \Z^d$ with infinitesimal generator
\begin{align*}
	(\mathcal L_{\text{sc}}g)(x,z) &= \sum_{k = 1}^R \min\{\lambda_k(x),\widetilde \lambda_k(z)\} ( g(x+\zeta_k,z+\zeta_k) - g(x,z))\\
	&\hspace{.1in} + \sum_{k = 1}^R(\lambda_k(x) -  \min\{\lambda_k(x),\widetilde \lambda_k(z)\} ) ( g(x+\zeta_k,z) - g(x,z))\\
	&\hspace{.1in}+ \sum_{k = 1}^R(\widetilde\lambda_k(z) -  \min\{\lambda_k(x),\widetilde \lambda_k(z)\} ) ( g(x,z+\zeta_k) - g(x,z)),
\end{align*}
where $g: \Z^d\times \Z^d\to \R$ is any bounded function with compact support.

\subsection{Common random numbers}

In the common random numbers (CRN) coupling, we simply simulate the embedded discrete time Markov chain for each process concurrently with the exponential holding time for each transition.  The processes $X$ and $Z$ are then coupled by using (i) the same stream of random variables for the generation of the embedded discrete time chain, and (ii) the same stream of random variables for the exponential holding times.  

More explicitly, let $\{U_i\}_{i=0}^\infty$
be a sequence of uniform random variables over the interval $[0,1]$, and let $\eta: \R_{\ge 0}^R \times [0,1]\to \{\zeta_1,\dots,\zeta_R\}$ be defined via 
\[
	\eta (c_1, ...., c_R, u)  = \zeta_k  ~~~ \textrm{ if } ~~~ \frac{{\sum_{i=1}^{k-1} c_i}}{\sum_{i=1}^R c_i}  \leq
  u  < \frac{{\sum_{i=1}^k c_i}}{\sum_{i=1}^R c_i},
 \]
which is a categorical random variable parametrized by $c_1,...., c_R$.  
Also, let us denote
\begin{equation}
	\lambda_0(x) = \sum_{k=1}^R \lambda_k(x) \quad \text{and} \quad  \widetilde \lambda_0(x) = \sum_{k=1}^R \widetilde \lambda_k(x).
\end{equation}
Then for a common unit-rate poisson process $Y$,  which will determine the exponential holding times, we consider the following system: 
\begin{align}
\begin{split}
R_X(t) &= Y\left( \int_0^t \lambda_0(X_{\text{crn}}(s)) ds \right)  \\
R_Z(t) &= Y\left( \int_0^t \widetilde \lambda_0(Z_{\text{crn}}(s)) ds \right) \\ 
X_{\text{crn}}(t) &= X_{\text{crn}}(0) + \int_0^t \eta (\lambda_1(X_{\text{crn}}(s-)),\dots, \lambda_R(X_{\text{crn}}(s-)), U_{R_X(s-)} )   dR_X(s)  \\ 
Z_{\text{crn}}(t) &= Z_{\text{crn}}(0) + \int_0^t \eta (\widetilde \lambda_1(Z_{\text{crn}}(s-)),\dots,\widetilde\lambda_R(Z_{\text{crn}}(s-)),
U_{R_Z(s-)} )   dR_Z(s),
\end{split}
\end{align}
where we note that the processes shared not just the Poisson process $Y$, but also the sequence of uniform $[0,1]$ random variables $\{U_i\}_{i = 0}^\infty$.
The solution to this system exists and is unique by construction \cite{AndKurtz2011,Gill76,Gill77}.  We  note that while the
representations are different, the marginal processes $X_{\text{crn}}$ and $X_{\text{sc}}$ have the same distribution, while the coupled processes $(X_{\text{crn}}, Z_{\text{crn}})$  and $(X_{\text{sc}}, Z_{\text{sc}})$  obviously do not.

\subsection{Common reaction path coupling and the local common reaction path coupling}

The common reaction path (CRP) coupling arises by simply noting that we may couple the processes \eqref{eq:RTC_X} and \eqref{eq:RTC_Z} via the Poisson processes $\{Y_k\}$.  That is, in the CRP coupling $(X_{\text{crp}},Z_{\text{crp}})$ satisfies 
\begin{align}
\begin{split}
X_{\text{crp}}(t) = X_{\text{crp}}(0) + \sum_{k=1}^R Y_{k} \left(
 \int_0^t \lambda_k(X_{\text{crp}}(s)) ds \right)  \zeta_k \\
Z_{\text{crp}}(t) = Z_{\text{crp}}(0) + \sum_{k=1}^RY_{k} \left( \int_0^t \widetilde
  \lambda_k(Z_{\text{crp}}(s) ) ds
\right)\zeta_k,
\end{split} \label{CRP}
\end{align}
where the $Y_k$ are  independent unit-rate Poisson processes. 

Numerical experiments have shown that this coupling is significantly tighter than the CRN coupling, in that it produces a lower variance between the coupled processes, for many situations \cite{AndCFD2012,Khammash2010,Srivastava2013}.  However, the variance between the processes often  increases substantially as $t$ grows.  In fact, the variance of the relevant estimators oftentimes approaches that of independent realizations of $X$ and $Z$ as $t$ grows towards infinity \cite{AndCFD2012,Srivastava2013}.
 We postulate that the variance of the CRP coupling increases in this manner because of its
 inability to fix  a ``decoupling'' once it occurs.  To understand this heuristically, suppose that  given $X_{\text{crp}}(t_0)$ and $Z_{\text{crp}}(t_0)$ for some $t_0>0$ we also have
\begin{align}
 \int_0^{t_0} \widetilde
  \lambda_k(Z_{\text{crp}}(s)) ds  \ll  \int_0^{t_0} \lambda_k(X_{\text{crp}}(s))  ds    \label{trap} 
\end{align}
for all $k$.  Then  the time and type of the next jump of $X_{\text{crp}}$ is nearly uncorrelated from the time and type of the next jump of $Z_{\text{crp}}$.  This is true even if $X_{\text{crp}}(t_0)$ and $Z_{\text{crp}}(t_0)$ are very close or even equal.  %Once the process stumbles into the situation of  $  \lambda_k(X_{\text{crp}}(s)) \ll \widetilde \lambda_k(Z_{\text{crp}}(s))$, then the probability of escaping the trap \eqref{trap} in any finite time can be very small, depending on the choice of the intensity functions.\\
 This problem does not occur with the split coupling since the next jump times of
$X$ and $Z$ are always correlated via the  counting processes with  
intensity 
\[
	\lambda_k(X_{\text{sc}}(s)) \wedge \widetilde \lambda_k(Z_{\text{sc}}(s)). 
\]   
%NOTE FROM DAVE:  I am not sure this should be here.  Confuses the point.

%Further, assume that $X(0) = Z(0)$ and define
%$\tau_0 = 0$ and  $\tau_i = \inf\{s > \tau_{i-1} + H_{i-1} :
%X(s) = Z(s)\}$ for $i \ge 1$, where $H_{i}$ is
%the holding time in the state
%$(X(\tau_i),Z(\tau_i))$.  Note that the
%$\tau_i$ are the (stopping) times at which the processes
%$X$ and $Z$ come back together, or ``recouple.'' %We see that
%                                %the processes $X_{\text{sc}}$ and
%                                %$Z_{\text{sc}}$ fully recouple at the
%                                %times $\tau_i$, a property not
%                                %enjoyed by the CRP coupled processes
%                                %$(X_{\text{crp}},Z_{\text{crp}})$. 
%We see that $\tau_i$ is much easily desribed for the split coupled
%processes than for the CRP coupled processes. In particular, 
%$$\tau_i \Big \vert_{\tau_i >  t} \in \sigma\{ (X(t), Z(t)) \}  $$
%is a property of $\tau_i$ that is enjoyed only for the split coupled
%processes.  
 
 The above discussion motivates us to consider the following modification to the CRP coupling.  We discretize $[0,T]$ into multiple subintervals.  For each such subinterval we generate the coupled processes using a new set of independent unit-rate Poisson processes and initial conditions given by the values of the processes at the terminal time of the previous subinterval.   Note that if the processes $X_{\text{crp}}$ and $Z_{\text{crp}}$ are equal to each other at a transition between subintervals, then  the processes will have recoupled.
 %The above discussion motivates us to reconsider the CRP coupling and  ``reset'' the Poisson processes at small time
%intervals in order to overcome the problem of $X_{\text{crp}}$ and $Z_{\text{crp}}$ can fully decouple from each other. In particular, we wish to ensure that whenever the processes $X_{\text{crp}}(s)$ and $Z_{\text{crp}}(s)$ are equal, the processes recouple.  
We will elaborate on this strategy. 
Let $\pi = \{0 = s_0 < s_1  \cdots < s_n= T\} $  be a partition of
$[0,T]$. Also let $\{Y_{km}: k = 1,\dots,R, m =0,1,2,\dots~ \}$ be a
set of independent, unit-rate Poisson processes.  
Then we define the local-CRP coupling over $[0, T]$ with  respect to $\pi$  as the solution of 

\begin{align}
\begin{split}
 X_{\text{crp}}^\pi(t)  &= X(0) + \sum_{k =1}^R \sum_{m=0}^{\infty}  Y_{km} \left(
  \int_{t \wedge  s_m}^{t \wedge s_{m+1}} \lambda_k(X^{\pi}_{\text{crp}}(s) ) ds  \right)
\zeta_k \\
 Z_{\text{crp}}^\pi(t) &= Z(0)+ \sum_{k =1}^R \sum_{m=0}^{\infty}  Y_{km} \left(
   \int_{t \wedge  s_m}^{t \wedge s_{m+1}}  \widetilde \lambda_k(Z^{\pi}_{\text{crp}}(s)) ds  \right)
\zeta_k .
\end{split} \label{localCRP}
\end{align}

We remark that, irrespective of $\pi$, the marginal distribution of
$X_{\text{crp}}^\pi$ is the same as that of $X$, our process of interest, and the same goes for $Z^{\pi}_{\text{crp}}$ and $Z$.  
Also, when $\pi$ is a trivial partition with $n =1$, the coupling  \eqref{localCRP}
is precisely the CRP coupling of \eqref{CRP}.  In the next section, we
will consider the limit of the the family of local-CRP couplings as  $n \to \infty$ and prove that under reasonable conditions the coupled processes converge weakly to the processes coupled via the split coupling \eqref{eq:split_coupling}.

\section{Limit of the local-CRP coupling} 
\label{sec:Analysis}

We begin this section by specifying some notation.  First, when $X$ and $Z$ are stochastic processes built on the probability space $(\Omega,\mathcal F,P)$, we denote by $X(s,\omega)$ the process $X$  evaluated at time $s$ for a given choice $\omega \in \Omega$.  Further, by $(X, Z)(s,\omega)$ we mean $( X(s, \omega), Z(s,\omega)) $,  a vector of random
variables evaluated at time $s$. As is usual, we will often omit $\omega$ from the notation when no confusion is expected. Finally, 
when $\mathbf{t} = (t_1,\dots,t_K)$ is a $K$ dimensional vector of times points, we denote
\[
	X(\mathbf{t})  =  [X(t_1), \dots, X(t_K)].
\]
Also, throughout the section, we assume that $X(0) = Z(0)$.  

\subsection{Weak convergence of finite dimensional distributions}

We will first articulate what we mean by taking $n\to \infty$ in
the context of the last section.  
\begin{defn}
Let $\pi_n = \{0 = s_0 \le s_1  \le \cdots \le s_n= T\} $ be a partition of
$[0,T]$. For $m \in \{0,\dots,n-1\}$  let
$$\Delta_m \pi_n= s_{m+1} - s_m. $$ 
The mesh of $\pi_n$ is defined as 
\[
	\text{mesh}(\pi_n) \eqdef \max\{ \Delta_m \pi_n : m \in \{0,\dots,n-1\}\}.
\]
\end{defn}

Supposing that $\text{mesh}(\pi_n) \to  0$ as $n \to \infty$, the limit of  interest to us is the weak limit of $(X^{\pi_n}_{\text{crp}},Z^{\pi_n}_{\text{crp}})$ as $n \to \infty$.   We begin with Proposition \ref{Theorem} showing the weak convergence of $X_{\text{crp}}^{\pi_n}$ to $X_{\text{sc}}$ over finite coordinates as $n\to \infty$.  In Subsection \ref{sec:weak_conv} we prove weak convergence at the process level.

\begin{prop}
Suppose that neither of the nominal processes $X,Z$ are explosive and let $(X_{\text{sc}}(t), Z_{\text{sc}}(t))$ be coupled in the way of \eqref{eq:split_coupling}.  Let 
\[
	\pi_n = \{0 = s_0 \leq s_1 \leq \cdots \leq s_{n} =T  \}
\]
 be a sequence of partitions such that $\text{mesh} (\pi_n) \to 0$, as $n \to \infty$, and for each $n$ let $(X^{\pi_n}_{\text{crp}}(t), Z^{\pi_n}_{\text{crp}}(t))$ be coupled in the way of \eqref{localCRP}.  
  Then for any $K \in \ZZ_{\ge 0}$ and 
$\mathbf{t} \in [0,T]^K$,  and any bounded Lipshitz $f: (\RR^d \times \RR^d)^K \to \RR$,  
\noindent  
\[
	\E[f((X^{\pi_n}_{\text{crp}} ,  Z^{\pi_n}_{\text{crp}})( \mathbf{t}))] \rightarrow \E[f((X_{\text{sc}}, Z_{\text{sc}}) (\mathbf{t}))], \quad \text{as } n \to \infty.
\]
  \label{Theorem} 
\end{prop}

%Before the proof, we would like to add an additional set of notations. 
We will briefly outline the proof of  \ref{Theorem}. For a fixed $n$, let
\begin{align}
\{ Y_{ikm}^n; ~~i = 1,2,3, \quad k = 1,\dots,R, \quad m = 0,1,2,... \}  \label{OrdSpace}
\end{align}
and
\begin{align}
\{ Y_{km}^n; ~~ k = 1,\dots,R, \quad m = 0,1,2,\dots \} \label{OrdSpace2} 
\end{align}
be two sets of independent unit-rate Poisson processes.  At this point, we do not make any assumption on the 
correlation between the processes in the set \eqref{OrdSpace} and the processes in the set
\eqref{OrdSpace2}, except to note that they will not be independent.  In fact, we will construct the Poisson processes of \eqref{OrdSpace}  as functions of the Poisson processes of \eqref{OrdSpace2}. % We will couple them  in a systematic way. 
For now,  simply consider the processes built using the Poisson processes of \eqref{OrdSpace}
\begin{align}
\begin{split}
X_{\text{sc}}^{\pi_n}(t)= X_{\text{sc}}(0)  + &\sum_{m = 0}^{\infty}\sum_{k=1}^R \Bigg\{ Y_{1km}^n \left(
 \int_{s_m \wedge t}^{s_{m+1}\wedge t}  r_{1k}(\lambda_k, \widetilde \lambda_k, X_{\text{sc}}^{\pi_n}, Z_{\text{sc}}^{\pi_n})(s) ds \right)  \\
&+   Y_{2km}^n
\left( \int_{s_m \wedge t}^{s_{m+1}\wedge t} r_{2k}(\lambda_k, \widetilde \lambda_k, X_{\text{sc}}^{\pi_n}, Z_{\text{sc}}^{\pi_n})(s)  ds \right)
\Bigg\}  \zeta_k \\
Z_{\text{sc}}^{\pi_n}(t) = X_{\text{sc}}(0) + &\sum_{m = 0}^{\infty}\sum_{k=1}^R\Bigg\{ Y_{1km}^n \left(
  \int_{s_m\wedge t}^{s_{m+1}\wedge t}
  r_{1k}(\lambda_k, \widetilde \lambda_k, X_{\text{sc}}^{\pi_n}, Z_{\text{sc}}^{\pi_n})(s)  ds
\right) \\
& +  Y_{3km}^n \left( \int_{s_m \wedge t}^{s_{m+1}\wedge t} r_{3k}(\lambda_k, \widetilde \lambda_k, X_{\text{sc}}^{\pi_n}, Z_{\text{sc}}^{\pi_n})(s)  ds  \right)  \Bigg\}  \zeta_k,
\end{split} \label{ordinary2}
\end{align}
along with 
\begin{align}
\begin{split}
 X_{\text{crp}}^{\pi_n}(t)  &= X_{\text{crp}}(0) + \sum_{m=0}^{\infty} \sum_{k =1}^R Y_{km}^n \left(
  \int_{t \wedge  s_m}^{t \wedge s_{m+1}} \lambda_k(X^{\pi_n}_{\text{crp}}(s) ) ds  \right)
\zeta_k \\
 Z_{\text{crp}}^{\pi_n} (t) &= X_{\text{crp}}(0) + \sum_{m=0}^{\infty}  \sum_{k =1}^R Y_{km}^n \left(
   \int_{t \wedge  s_m}^{t \wedge s_{m+1}}  \widetilde \lambda_k(Z^{\pi_n}_{\text{crp}}(s) ) ds  \right)
\zeta_k,
\end{split} \label{localCRP2}
\end{align}
which are built with the Poisson processes \eqref{OrdSpace2}.
Note that % the generator of $(X_{\text{sc}},Z_{\text{sc}})$ is the same as that of $(X_{\text{sc}}^{\pi_n}, Z_{\text{sc}}^{\pi_n})$, so
 $(X_{\text{sc}},Z_{\text{sc}}) \overset{dist}{=} (X_{\text{sc}}^{\pi_n}, Z_{\text{sc}}^{\pi_n})$ irrespective of $n$. %, and irrespective of the coupling.  
 The construction we will employ will allow us to conclude that 
%We will couple \eqref{OrdSpace2} and \eqref{OrdSpace} while constructing \eqref{ordinary2} from \eqref{localCRP2} in the following manner. 
%We will first generate $(X_{\text{crp}}^{\pi_n}, Z_{\text{crp}}^{\pi_n})$ up to time $T$. This will give us a realization of \eqref{OrdSpace2} up to some random time.  We will then construct \eqref{OrdSpace} using the observed \eqref{OrdSpace2}, and use \eqref{OrdSpace} to construct $(X_{\text{sc}}^{\pi_n}, Z_{\text{sc}}^{\pi_n}).$ 
%We follow these procedures in such a way that the constructed
 $(X_{\text{sc}}^{\pi_n}, Z_{\text{sc}}^{\pi_n})$ and 
$(X_{\text{crp}}^{\pi_n},Z_{\text{crp}}^{\pi_n})$ satisfy
\begin{align}
\begin{split}
\lim_{n \to \infty}P\left( \max_{i\in\{0,...,K \}}
|(X_{\text{sc}}^{\pi_n}(t_i),Z_{\text{sc}}^{\pi_n}(t_i)) -  (X_{\text{crp}}^{\pi_n}(t_i),
Z_{\text{crp}}^{\pi_n}(t_i))| > \gamma\right) = 0
\end{split} \label{Couple}
\end{align}
for any $\gamma > 0$. We can then appeal to a standard Portmanteau
type argument to finish the proof of Proposition \ref{Theorem}: let $\epsilon >0$, and consider any bounded
continuous map 
$f: (\RR^d \times \RR^d)^K \to \RR$ 
with Lipshitz constant $L$.  Then
\begin{align*}
\begin{split}
&|\E f( (X_{\text{sc}}, Z_{\text{sc}})(\textbf{t})  - \E f(
(X_{\text{crp}}^{\pi_n},Z_{\text{crp}}^{\pi_n})(\textbf{t})  |  \\
&~~~~~~~~~= | \E f( (X_{\text{sc}}^{\pi_n},
Z_{\text{sc}}^{\pi_n})(\textbf{t})- \E f( (X_{\text{crp}}^{\pi_n}, Z_{\text{crp}}^{\pi_n})(\textbf{t}) |\\
&~~~~~~~~~\leq L \E [  |(X_{\text{sc}}^{\pi_n},Z_{\text{sc}}^{\pi_n}) (\textbf{t}) -
(X_{\text{crp}}^{\pi_n} , Z_{\text{crp}}^{\pi_n}) (\textbf{t})| ] \\
&~~~~~~~~~\leq L K \gamma +  L~ P(\max_{i=0,...,K } |(X_{\text{sc}}^{\pi_n}(t_i),Z_{\text{sc}}^{\pi_n}(t_i)) -  (X_{\text{crp}}^{\pi_n}(t_i), Z_{\text{crp}}^{\pi_n}(t_i))| > \gamma).  
\end{split}
\end{align*}
We can first choose $\gamma < \epsilon/(2LK)$.  With this $\gamma$ fixed, we
may choose $n$ large enough so that  the second piece can be bounded by
$\epsilon/2$, and the claim is achieved. 

\vspace{.1in}

We must still describe the specific construction alluded to above that will allow us to conclude \eqref{Couple}.  
%We will describe the coupling of 
%$(X_{\text{crp}}^{\pi_n},Z_{\text{crp}}^{\pi_n})$ and $(X_{\text{sc}}^{\pi_n},Z_{\text{sc}}^{\pi_n}) $ which will make \eqref{Couple} possible.  
For each $n$, let 
\begin{align}
\{ Y_{km}^n, Y_{ikm}^{n,aug}, i = 1,2,3, ~ k = 1,\dots,R, ~m = 0,1,2,\dots\},   \label{Space}
 \end{align}
 be independent unit-rate Poisson processes.
%$$\{ Y_{ikm}^n; ~~i = 1,2,3 ~~~ k = 1,...,R ~~~m = 0,...,N(n) \} $$
%As promised, let us 
We   generate $(X_{\text{crp}}^{\pi_n}, Z_{\text{crp}}^{\pi_n})$ up to time $T$ using the processes $Y_{km}^n$ according to \eqref{localCRP2}.  We now turn our attention to constructing the required independent unit-rate Poisson processes $Y_{ikm}^n$, and the coupled processes $(X_{\text{sc}}^{\pi_n}, Z_{\text{sc}}^{\pi_n})$ built using them according to \eqref{ordinary2}.

Inductively arguing on $m$, suppose we have already generated $(X_{\text{sc}}^{\pi_n}, Z_{\text{sc}}^{\pi_n})$ given by \eqref{ordinary2} up to time $s_m\ge 0$.  We further suppose that we have constructed the relevant Poisson processes $Y_{ik\tilde m}^n$ for all $\tilde m < m$.  We must now describe how to construct $Y_{ikm}^n$ for each valid pair $(i,k)$.  We  define the following random  times for each $i\in \{1,2,3\}$ and $k\in \{1,\dots,R\}$:
\begin{align}\label{eq:cT_gen}
\begin{split}
&\cT_{ikm} \eqdef r_{ik}(\lambda_k, \widetilde \lambda_k, X_{\text{sc}}^{\pi_n}, Z_{\text{sc}}^{\pi_n})(s_m)  \cdot  \Delta_m(\pi_n)
\end{split}
\end{align} 
% \begin{align}
% \begin{split}
% &\cT_{2km} = (\lambda_k(Z_{s_m})  - \widetilde \lambda_k(X_{s_m})) \cdot \Delta_m(\pi_n)  1_{\lambda(X_{s_m}) \geq  {\lambda(Z_{s_m}) }}  \\
% &T_{2km}^{\text{sc}} =  \tau_m(\widetilde \lambda_k, Z) -
%   \tau_m(\lambda_k \wedge \widetilde \lambda_k, X, Z), ~~ \\
% \end{split}
% \end{align} 
% \begin{align}
% \begin{split}
% \cT_{3km} &=  \lambda_k(X_{s_m})  \wedge \widetilde \lambda_k(Z_{s_m}) \cdot \Delta_m(\pi_n)  \\
% T_{3km}^{\text{sc}} &=  \tau_m(\lambda_k \wedge \widetilde \lambda_k, X, Z) \\
% \end{split}
% \end{align}
and
\begin{align*}
T_{km}^{\text{crp}} \eqdef \left(\int_{s_m}^{s_{m+1}}  \lambda(X_{\text{crp}}^{\pi_n}(s) ) ds\right) \vee
\left(\int_{s_m}^{s_{m+1}}  \widetilde \lambda (Z_{\text{crp}}^{\pi_n}(s) ) ds\right),
\end{align*}
where, as usual, $a \vee b \eqdef \max\{a,b\}$, and we recall that $(X_{\text{crp}}^{\pi_n}, Z_{\text{crp}}^{\pi_n})$ has already been generated up to time $T$.
For notational clarity we refrain from using $n$ in the notation above for the random times. 
% Since we will create the coupling between \eqref{Space} and \eqref{OrdSpace} for each fixed $n$, we have omitted $n$ from the notation for  $\cT_{ikm}$, 
%$T_{ikm}^{\text{sc}}$ \textcolor{red}{DAVE:  Did I miss where this was defined?  I don't think $T_{ikm}^{\text{sc}}$ should be here yet}, and $T_{km}^{\text{crp}} $ above to avoid notational overflow. 
%Inductively arguing on $m$, suppose that we have seen
%$(X_{\text{sc}}^{\pi_n},Z_{\text{sc}}^{\pi_n})$ up to time $s_m$
%in \eqref{ordinary2}. At this point, we have not generated
%$Y_{ik\tilde m}$ for any $\tilde m \ge m$, and we now describe how to construct $Y_{ikm}$ for each valid pair $(i,k)$.
 We now define 
$Y_{1km}^n$ in the following manner% \textcolor{red}{(MASO:  Do you still want to remove the figure? )}  DAVE: I think it is clear enough as is.  The picture actually confused even me, so I think it is best left out.
\begin{align*}
\begin{split}
Y_{1km}^n(u) &=Y_{km}^n(u) \hspace{1.77in} \textrm{for}~~ u\leq \cT_{1km} \\[1ex]
Y_{1km}^n(u) &=  Y_{1km}^n (\cT_{1km} ) + Y_{1km}^{n,aug}(u-\cT_{1km}) ~~~~~~ \textrm{for}~~ u >  \cT_{1km}. 
\end{split}
\end{align*}
Having defined $Y_{1km}^n$, we turn to the construction of $Y_{2km}^n$ and $Y_{3km}^n$.  The construction is based on which of two of the following cases hold.
\begin{enumerate}[1.]
	\item If   $ \lambda(Z_{\text{sc}}^{\pi_n}(s_m)) \le \lambda(X_{\text{sc}}^{\pi_n}(s_m))$, then let $Y_{2km}^n$ satisfy
\begin{align*}
\begin{split}
Y_{2km}^n(u) &= Y_{km}^n(u + \cT_{1km}) -  Y_{km}^n(\cT_{1km})   \hspace{0.77in} \textrm{for}~~
u \leq \cT_{2km} \\[1ex]
Y_{2km}^n(u) &= Y_{2km}^n( \cT_{2km})  + Y_{2km}^{n,aug}(u-\cT_{2km}) \hspace{.56in} \textrm{for}~~ u > \cT_{2km}, 
\end{split}
\end{align*} 
and let $Y_{3km}^n(u) = Y_{3km}^{n, aug}(u)$ for all $u\ge 0$.

\item If   $ \lambda(Z_{\text{sc}}^{\pi_n}(s_m)) > \lambda(X_{\text{sc}}^{\pi_n}(s_m))$, then let $Y_{3km}^n$ satisfy
\begin{align*}
\begin{split}
Y_{3km}^n(u) &= Y_{km}^n(u + \cT_{1km}) -  Y_{km}^n(\cT_{1km})   \hspace{0.77in} \textrm{for}~~
u \leq \cT_{3km} \\[1ex]
Y_{3km}^n(u) &= Y_{3km}^n( \cT_{3km})  + Y_{3km}^{n,aug}(u-\cT_{3km}) \hspace{.56in} \textrm{for}~~ u > \cT_{3km}, 
\end{split}
\end{align*} 
and let $Y_{2km}^n(u) = Y_{2km}^{n, aug}(u)$ for all $u\ge 0$.
\end{enumerate}

% the Next, if $ \lambda(Z_{\text{sc}}^{\pi_n}(s_m)) < \lambda(X_{\text{sc}}^{\pi_n}(s_m))$, let $Y_{1km}^n$ be the jump process such that 
%\begin{align}
%\begin{split}
%Y_{2km}^n(s) = Y_{km}^n(s + \cT_{3km}) -  Y_{km}^n(\cT_{2km})   &~~~~~ \textrm{for}~~
%s \leq \cT_{1km} \\
%Y_{2km}^n(s) - Y_{2km}^n( \cT_{1km})  = Y_{2km}^{n,aug}(s) &~~~~~\textrm{for}~~s \geq \cT_{2km} 
%\end{split}
%\end{align} 
%% \begin{align}
%% \begin{split}
%% Y_{1km}^n(s) - Y_{1km}^n( \cT_{1km})  = Y_{1km}^{n,aug}(s)
%% \end{split}
%% \end{align}
%and $Y_{3km}^n = Y_{3km}^{n, aug}$. If $ \lambda(Z_{\text{sc}}^{\pi_n}(s_m)) \geq \lambda(X_{\text{sc}}^{\pi_n}(s_m))$, then we will interchange the role of $2km$ with $3km$.
%
%\begin{figure}[h]
%\begin{center}
%\begin{tikzpicture}{
%\draw (0,0) node {$Y_{km}$}; 
%\draw[->, line width = 2pt] (1,0) -- (15,0);
%\draw (3,0.6) node {$\cT_{1km}$} ; 
%\draw (0,0.4) node {$Y_{1km}$}; 
%\draw(1,0.4) -- (5,0.4);
%
%\draw[->, line width = 2pt] (5,0.4) -- (10,3);
%\draw (7,2) node {$Y_{1km}^{aug}$} ; 
%
%\draw (6,0.4) node {$Y_{2km}$}; 
% \draw(6.5,0.4) -- (10,0.4);
%\draw (8,0.6) node {$\cT_{2km}$} ;
%
%\draw (12,2) node {$Y_{2km}^{aug}$} ;  
%\draw[->, line width = 2pt] (10,0.4) -- (15,3);
%
%\draw(5, -1.6) node{$Y_{3km}^{aug}$};
%\draw (0,-2) node {$Y_{3km}$}; 
%\draw[->, line width = 2pt](1, -2) -- (10, -2);   
%}\end{tikzpicture}
%\end{center}
%\caption{A pictorial image of coupling \eqref{OrdSpace} to
%  \eqref{Space} in the case of
%  $ \lambda(Z_{\text{sc}}(s_m)) < \lambda(X_{\text{sc}}(s_m))$. }
%\end{figure}

\noindent Note that the strong Markov property guarantees that the processes $\{Y_{ikm}^n\}$ so constructed are independent, unit-rate Poisson processes. We then generate $(X_{\text{sc}}^{\pi_n}, Z_{\text{sc}}^{\pi_n})$ between times $s_m$ and $s_{m+1}$ according to \eqref{ordinary2} with the processes $\{Y_{ikm}^n\}$.  Note that in so doing, we have also created a coupling between  $(X^{\pi_n}_{\text{sc}}, Z^{\pi_n}_{\text{sc}})$ and $(X^{\pi_n}_{\text{crp}}, Z^{\pi_n}_{\text{crp}})$.

Note that for each $i, k$, and $m$, the value $\cT_{ikm}$ as defined in \eqref{eq:cT_gen} is an
approximation to 
$$T_{ikm}^{\text{sc}} \eqdef \int_{s_m}^{s_{m+1}}  r_{ik}(\lambda_k,
\widetilde \lambda_k, X_{\text{sc}}^{\pi_n},Z_{\text{sc}}^{\pi_n})(s) \, ds.$$
We would like to make a few observations about this approximation before proceeding further. 
\begin{lemma}
Fix $n$, and  let $m \in \{0,1,\dots\}$.  
If  
\begin{align}
\sum_{k=1}^R \sum_{i=1}^3  Y_{ikm}^n(\cT_{ikm} \vee T_{ikm}^{\text{sc}}) =1 \label{HalfCondition}
\end{align}
 then 
there is a unique $j \in \{1,2,3\} $ and $\ell \in \{1, ..., R \} $ for
which
 \[
 	Y_{j \ell m}^n(\cT_{j \ell m}  \wedge T_{j \ell m}^{\text{sc}}) =1.
\]
  \label{LetmeGo}
\end{lemma} 
Note the difference between $\wedge$ and $\vee$ in the above statement.   
\begin{proof}
For each $(i,k)$,  define
\[
	Q_{ik}(t) \eqdef Y_{ikm}^n \left( \int_{s_m}^{t+s_m} r_{ik}(\lambda_k, \widetilde \lambda_k, X_{\text{sc}}^{\pi_n},
  Z_{\text{sc}}^{\pi_n})(s) ds \right),
  \]
for $t \ge 0$. 
Note that \eqref{HalfCondition} implies that $Y_{j \ell m}^n(\cT_{j \ell m}
\vee T_{j \ell m}^{\text{sc}}) = 1$ for some $j$ and $\ell$ and $Y_{i k
  m}^n(\cT_{i k m}
\vee T_{ik m}^{\text{sc}}) = 0$  for all $(i,k) \neq (j, \ell)$. In particular, this implies $Q_{j \ell}$ is the first one among the set of counting processes $\{Q_{ik} \} $ to jump. (This follows since  for all $(i,k)$, $r_{ik}(\lambda_k, \widetilde \lambda_k, X_{\text{sc}}^{\pi_n},
  Z_{\text{sc}}^{\pi_n})(s)$ will not change from $r_{ik}(\lambda_k, \widetilde \lambda_k, X_{\text{sc}}^{\pi_n},
  Z_{\text{sc}}^{\pi_n})(s_m)$ until the first jump of $(X_{\text{sc}}^{\pi_n},
  Z_{\text{sc}}^{\pi_n})$ during $s>s_m$.)  By the definitions of $\cT_{j\ell m}$ and $T^{\text{sc}}_{j \ell m}$, it easily follows that $Y_{j\ell m}^n (\cT_{j\ell m} \wedge T^{\text{sc}}_{j\ell m} = 1$.  It is trivial that no other $(i,k)$ pair can satisfy this relation.
  \end{proof}

%   If the first jump time of $Q_{j \ell}$ is at time $t= \alpha$, then $Y_{j \ell m}^n$ must mark the first 
%jump at some $t_0$ that can be represented as 
%\begin{align}
%t_0 = r_{j \ell} (\lambda_k, \widetilde \lambda_k, X_{\text{sc}}^{\pi_n},
%  Z_{\text{sc}}^{\pi_n})(s_m) \alpha. \label{repn}
%\end{align}
%If $\alpha > \Delta_n{\pi_n}$, then by definition
%  $Y_{j \ell m}^n(\cT_{j \ell m} \vee T_{j \ell m}^{\text{sc}}) = 0$ and that
%  will be a contradition. Hence $\alpha
%  \leq \Delta_n{\pi_n}$ necessarily and  $t_0 \leq T_{j \ell m}^{\text{sc}}$. 
%Also, from the \eqref{repn},  $t_0 \leq \cT_{j \ell  m}$ trivially. 
%\end{proof}

The following is an analogue to Lemma \eqref{LetmeGo}.
\begin{lemma}
If $( X_{\text{crp}}^{\pi_n}, Z_{\text{crp}}^{\pi_n})(s_m) = ( X_{\text{sc}}^{\pi_n}, Z_{\text{sc}}^{\pi_n})(s_m)$ and
\[
	\sum_k Y_{km}^n \left(  \left( \sum_{i=1}^3 \cT_{i k m} \right)    \vee T^{\text{crp}}_{km}    \right)  = 1
\]
then there is a unique $j$ for which  
\[
	Y_{j m}^n \left( \left( \sum_{i=1}^3 \cT_{i j m} \right)   \wedge  T^{\text{crp}}_{j m} \right)  =1. 
 \] 
Further, the first jump time of the Poisson process $Y_{j m}^n$ occurs at some $t_0$ satisfying	
\[
	t_0 <  \left( \lambda_{j}(X_{\text{crp}}^{\pi_n}(s_m)) \vee \widetilde
\lambda_{j}(Z_{\text{crp}}^{\pi_n}(s_m)) \right) \Delta_m.
\]
 \label{LetmeGo2}
\end{lemma}
\begin{proof}
Because the two processes are equal at time $s_m$, we have that 
\[
	 \sum_{i=1}^3 \cT_{i k m}   =  \left ( \lambda_{k_0}(X_{\text{crp}}^{\pi_n}(s_m)) \vee \widetilde
\lambda_{k_0}(Z_{\text{crp}}^{\pi_n}(s_m)) \right) \Delta_m.
\]
As neither $Z_{\text{crp}}^{\pi_n}$ nor $X_{\text{crp}}^{\pi_n}$ changes until the first
firing of $Y_{j m}$, the claim follows. 
\end{proof}

Based on the last two observations, we have the following lemma which will be useful in proving Proposition \ref{Theorem}.

\begin{lemma}
Fix $n$ and suppose that, for a given path
of $(X^{\pi_n}_{\text{sc}},Z^{\pi_n}_{\text{sc}})(\omega),$ $(X^{\pi_n}_{\text{crp}}, Z^{\pi_n}_{\text{crp}})(\omega)$
coupled in the way we described above, 

\begin{align}
\begin{split}
H_{m,n}(\omega) \eqdef \sum_{k=1}^R  \max\left \{ \sum_{i=1}^3
  Y_{ikm}^n(\cT_{ikm} \vee T_{ikm}^{\text{sc}}),    Y_{km}^n\left(\left(
      \sum_{i=1}^3  \cT_{ikm}  \right) \vee T_{km}^{\text{crp}} \right) \right \} \leq 1, 
\end{split} \label{Condition}
\end{align}
for all $m$.
Then for all $m = 0,\dots, n$, 
  $$(X_{\text{sc}}^{\pi_n},Z_{\text{sc}}^{\pi_n})(s_m,\omega) = (X_{\text{crp}}^{\pi_n}, Z_{\text{crp}}^{\pi_n})(s_m,\omega)$$  \label{Proper}
\end{lemma}

\begin{proof}
We will omit $\omega$ in the expressions. We have 
\[
	(X^{\pi_n}_{\text{sc}}, Z^{\pi_n}_{\text{sc}})(s_0) =  (X_{\text{crp}}^{\pi_n}, Z_{\text{crp}}^{\pi_n})(s_0)
\]
by assumption. 
Arguing inductively, assume that 
\begin{align*}
\begin{split}
(X^{\pi_n}_{\text{sc}}, Z^{\pi_n}_{\text{sc}})(s_m) =  (X_{\text{crp}}^{\pi_n}, Z_{\text{crp}}^{\pi_n})(s_m).
\end{split} %\label{ASSUMPTION1}
\end{align*}
We will show that
$$(X^{\pi_n}_{\text{sc}}, Z^{\pi_n}_{\text{sc}})(s_{m+1}) =  (X^{\pi_n}_{\text{crp}}, Z^{\pi_n}_{\text{crp}})(s_{m+1})$$
when \eqref{Condition} holds. 
If  $H_{m,n} = 0$ for this $m$,
then 
\[
(X^{\pi_n}_{\text{sc}}, Z^{\pi_n}_{\text{sc}})(s_{m+1}) =  (X^{\pi_n}_{\text{sc}}, Z^{\pi_n}_{\text{sc}})(s_m) =  (X^{\pi_n}_{\text{crp}}, Z^{\pi_n}_{\text{crp}})(s_m) =  (X^{\pi_n}_{\text{crp}}, Z^{\pi_n}_{\text{crp}})(s_{m+1}),
\]
and there  is nothing to do. Therefore we consider the case in which $H_{m,n} =
 1$.
More specifically, suppose that for some $k_0$,  
$$\max\left \{ \sum_{i=1}^3  Y_{i k_0 m}^n(\cT_{i k_0 m} \vee T_{i k_0 m}^{\text{sc}}),
  Y_{k_0 m}^n\left( \left(\sum_{i=1}^3  \cT_{i k_0 m} \right)\vee T_{ k_0 m}^{\text{crp}} \right) \right \}
=1.$$
This means that, by condition \eqref{Condition},
$$\max\left \{ \sum_{i=1}^3  Y_{i k m}^n (\cT_{i \ell m} \vee T_{i
    k m}^{\text{sc}}),
  Y_{k m}^n\left( \left(\sum_{i=1}^3  \cT_{i k m} \right)\vee
    T_{k m}^{\text{crp}} \right) \right \}
=0$$
for all $k \neq k_0$.  
Combined with Lemmas \ref{LetmeGo} and \ref{LetmeGo2}, these  conditions guarantee that  each of the processes $X_{\text{sc}}^{\pi_n},  Z^{\pi_n}_{\text{sc}},  X^{\pi_n}_{\text{crp}},  Z^{\pi_n}_{\text{crp}}$ jump precisely one time in the time interval $[s_m,s_{m+1}]$, and the jump happens according to reaction channel $k_0$ (see \cite{KoyamaThesis} for more details).  That is, we have
\[
	X_{\text{sc}}^{\pi_n}(s_{m+1}) =Z_{\text{sc}}^{\pi_n}(s_{m+1}) =X_{\text{crp}}^{\pi_n}(s_{m+1}) =Z_{\text{crp}}^{\pi_n}(s_{m+1}) = X_{\text{sc}}^{\pi_n}(s_{m}) + \zeta_{k_0},
\]
and we are done.
\end{proof}

It is not too difficult to see that  if $\lambda_k$ and $\widetilde \lambda_k$ are uniformly bounded for all $k$, then we can 
make the condition in Lemma \ref{Proper} hold with a probability greater than
$1-\epsilon$ for any $\epsilon>0$  by setting
$\text{mesh}(\pi_n)$ small enough.   Of course, we do not have such a uniform bound on the
intensity functions. Also,  note that Lemma \ref{Proper} does \textbf{not} imply that 
\[
	(X^{\pi_n}_{\text{crp}},Z^{\pi_n}_{\text{crp}})(t) = (X^{\pi_n}_{\text{sc}}, Z^{\pi_n}_{\text{sc}})(t) \text{  for } t \in [s_{m}, s_{m+1}],
	\]
 even if the conditions of the lemma are met, as the processes may (and most likely will) jump at slightly different times.   However, we trivially note that under the conditions of Lemma \ref{Proper},
 \begin{equation}\label{eq:245423}
 (X^{\pi_n}_{\text{crp}},Z^{\pi_n}_{\text{crp}})(t) = (X^{\pi_n}_{\text{sc}}, Z^{\pi_n}_{\text{sc}})(t) \text{ for all } t \in  [s_{m}, s_{m+1}]
 \end{equation}
 if neither $(X^{\pi_n}_{\text{crp}},Z^{\pi_n}_{\text{crp}})$ nor
$(X^{\pi_n}_{\text{sc}},Z^{\pi_n}_{\text{sc}})$ jump at all in $[s_m, s_{m+1}]$.

We are now in a position to prove Proposition \ref{Theorem}.

\begin{proof} [Proof of Proposition \ref{Theorem}]   
We first recall that $\textbf{t} = (t_1,\dots,t_K)$ for some $K\in \{1,2,\dots\}$.  Next, we define 
\[
	K_0^n \eqdef \{m \in \{ 0,...,n-1\} ~ ; ~ \{t_j\}_{j=1}^K \cap [s_{m} , s_{m+1}) \neq \emptyset \}.
\]
Fix $\epsilon>0$.  As we remarked around \eqref{Couple}, it suffices to show that, for large enough $n$, 
\[
	P\left( \max_{i=0,...,K }|(X_{\text{sc}}^{\pi_n}(t_i),Z_{\text{sc}}^{\pi_n}(t_i)) -  (X_{\text{crp}}^{\pi_n}(t_i),
Z_{\text{crp}}^{\pi_n}(t_i))| > 0\right) < \epsilon,
\]
where we converted the $\gamma$ in \eqref{Couple} to a zero as our processes take values in $\Z^d$.

We will resort to a localization
argument and take advantage of the fact that  $X$ and $Z$ are both
nonexplosive.
Let $M > 0$, and let $H_{m,n}$ be  defined as in Lemma \ref{Proper}. 
%We will define  on which \ref{Proper}  holds. 
Define
\begin{align}
\begin{split}
A_n(\mathbf{t}) &\eqdef \left\{ \omega :
 H_{m,n}(\omega) \leq 1 \textrm{~if } m
  \not \in K_0^n  \text{ and } H_{m,n}(\omega)= 0 \textrm{~ if } m \in K_0^n\right\},
\end{split}
\end{align}
%$$A_{nm}= \{  \textrm{only one of
%$Y_{ikm}$ fires in $[s_m, s_{m+1}]$  for the processes coupled with $\pi_n$ }   \}  $$  
and
\begin{align}
\begin{split}
B_{M,n} \eqdef &\{ \omega: \max\{ \sup_{s\leq T} \lambda_k(X^{\pi_n}_{\text{sc}}(s)), 
\sup_{s\leq T} \widetilde \lambda_k(Z^{\pi_n}_{\text{sc}}(s)),   \sup_{s\leq T}\lambda_k(X^{\pi_n}_{\text{crp}}(s)),
 \sup_{s\leq T}\widetilde \lambda_k(Z^{\pi_n}_{\text{crp}}(s)\}  \leq M  \}.
\end{split}
\end{align}
Note that by the non-explosivity of the processes,  the supremums are achieved everywhere they appear above.
By Lemma \ref{Proper} and the arguments in and around \eqref{eq:245423},   we have that
\[
	A_n(\mathbf{t})  \subset \{(X^{\pi_n}_{\text{sc}},Z^{\pi_n}_{\text{sc}})(\mathbf{t}) = (X^{\pi_n}_{\text{crp}}, Z^{\pi_n}_{\text{crp}})(\mathbf{t})\}.
\]
Therefore 
\begin{align}
P( (X^{\pi_n}_{\text{sc}},Z^{\pi_n}_{\text{sc}})(\mathbf{t}) \neq (X^{\pi_n}_{\text{crp}}, Z^{\pi_n}_{\text{crp}})(\mathbf{t}) ) &\leq P(A_n^C(\mathbf{t})) \notag \\
&= P(A_n^C(\mathbf{t}) \cap B_{M,n}) +  P(A_n^C (\mathbf{t})\cap B_{M,n}^C). \label{Ineq}
\end{align}
We handle the two pieces on the right hand side of \eqref{Ineq} separately.

For the second term in  \eqref{Ineq}, we first note that 
\begin{align*}
\begin{split}
B_{M,n}^C \subset \{\sup_{s\leq T} \lambda_k(X^{\pi_n}_{\text{sc}} (s)) >M\} &\cup \{\sup_{s\leq
  T}\widetilde \lambda_k(Z^{\pi_n}_{\text{sc}}(s)) > M \}  \cup \\ 
&\{\sup_{s\leq T} \lambda_k(X^{\pi_n}_{\text{crp}}(s)) >M\}   \cup
\{\sup_{s\leq T}\widetilde \lambda_k(Z_{\text{crp}}^{\pi_n}(s)) >M\}.
\end{split} 
\end{align*} 
Now, recall that the marginal distributions of $X^{\pi_n}_{\text{crp}}$ and $X^{\pi_n}_{\text{sc}}$ are the same as the
marginal distribution of $X$, and that the same goes for $Z^{\pi_n}_{\text{crp}}$ and $Z^{\pi_n}_{\text{sc}}$
compared with $Z$. Therefore, for all $n$ we have
\begin{align}
P(B_{M,n}^C)  \leq 2 \times \left [ P(\sup_{s \leq T} \{ \lambda_k(X_s) \} > M) +  P(\sup_{s \leq T}
 \{ \widetilde \lambda_k(Z_s) \} > M) \right ]  . \label{Monotone}
\end{align}
By the monotone convergence theorem and the fact that the processes
are all non explosive, the right hand side of \eqref{Monotone} will tend to $0$ as $M \to \infty$. 
Therefore, we can take $M$ large enough so
 that the second piece of  \eqref{Ineq}   is smaller than $\epsilon /2$.
We fix this $M$, and turn attention to the first term on the right hand side of \eqref{Ineq}.

We consider the localized version of $H$. In particular, for our fixed $M>0$ 
let 
\[
	H_{m,n}^M(\omega)  \eqdef \sum_{k=1}^R  \max\left \{ \sum_{i=1}^3
  Y_{ikm}^n(M\Delta_m(\pi_n)),    Y_{km}^n\left( 3M\Delta_m(\pi_n) \right)
\right \}.
\]
Then it is clear that, for any $q> 0, $ 
\[
	\left\{ \{ H_{m,n} > q\} \cap B_{M,n} \right\} \subset \left\{ \{
  H_{m,n}^M  > q\} \cap B_{M,n}  \right\} \subset \{H_{m,n}^M  > q\}
  \]
and therefore
\begin{align}
P(A_n^C(\mathbf{t}) & \cap B_{M,n}) \notag \\
&\leq P(  H_{m,n}^M > 1\textrm{ for some } m \not \in K_0^n \textrm{~~\textbf{OR}~~}  H_{m,n}^M >0 \textrm{ for some } m \in K_0^n
) \notag \\
& \leq \sum_{m \not \in K_0^n} P(H_{m,n}^M > 1) + \sum_{m \in K_0^n} P(H_{m,n}^M > 0).\label{eq:45873450}
\end{align} 
To handle these two pieces, we recall two basic facts pertaining to Poisson random variables.  First, if we denote by $W(\Lambda) \sim \text{Poisson}(\Lambda)$ then
\begin{align*}
P (W(\Lambda) > 1) = &1- \exp(-\Lambda) (1 +\Lambda)   \\
\leq & 1-  (1- \Lambda) (1 +\Lambda) \\
= &\Lambda^2,
\end{align*}
where we used the inequality $\exp(-x)\ge 1- x$.  Second, and using the same inequality,
\begin{align*}
P (W(\Lambda) > 0) &=  1- \exp(-\Lambda)  \leq \Lambda.
\end{align*}

Now note that 
\[
	P(  \{ H_{m,n}^M > q\}) \leq P(W(6RM \Delta_m (\pi_n)) > q ).
	\]
Hence, if $\text{mesh}(\pi_n) =  \delta_n$
then by the two facts above and \eqref{eq:45873450}, we have 
\begin{align}
P(A^C_n(\mathbf{t}) \cap B_M) &\leq \sum_{m \not \in K_0^n}  (6RM \Delta_m (\pi_n))^2  +   \sum_{m \in K_0^n}  (6RM \Delta_m (\pi_n)) \notag\\
&\leq   (6RM)^2 \delta_n \sum_{m \not \in K_0^n} 
\frac{ \Delta_m(\pi_n)}{\delta_n} \Delta_m(\pi_n)+    6RM|K_0^n|
\delta_n \notag \notag\\
&\leq  (6RM)^2 \delta_n T  + 6RM |K_0^n|
\delta_n,\label{Last}
\end{align}
where in the third inequality we used that $\frac{\Delta_m(\pi_n)}{\delta_n} < 1$, which follows by the definition of mesh. 
We can now take $n$ large enough so that \eqref{Last} is less
than $\epsilon/2$. Collecting the above, we may now conclude that for such $n$,  
$$P( |( X^{\pi_n}_{\text{sc}},Z^{\pi_n}_{\text{sc}})(\mathbf{t}) - (X^{\pi_n}_{\text{crp}}, Z^{\pi_n}_{\text{crp}})(\mathbf{t})| > 0 )  <
\epsilon,$$
 as required. 
\end{proof}

The following is an immediate corollary  to  Proposition \ref{Theorem}.
\begin{cor}
Let $s= \{s_0 < s_1 < s_2 < \cdots < s_{m_1}\} $ and  $t = \{t_0 < t_1 <
t_2 < \cdots < t_{m_2}\}$.  Let  $f_i : \RR^d \to \RR,$ $i = 0, \dots, m_1$, and $g_j : \RR^d \to \RR,$ $j = 0, \dots, m_2,$
%$$f_i : \RR^d \times \RR^d \to \RR ~~~ i = 0, ..., m $$ 
be bounded and continuous functions on $\RR^d$, and assume the
conditions set forth in Proposition \ref{Theorem}. Then 
%$$E\left[\prod_{i=0}^{m_1} f_i(X_{\text{crp}}^{\pi_n}(t_i)) \prod_{j=0}^{m_2} g_j(Z_{\text{crp}}^{\pi_n}(s_j))\right]
%\to  E\left[\prod_{i=0}^{m_1} f_i(X_{\text{sc}}(t_i)) \prod_{j=0}^{m_2} g_j(Z_{\text{sc}}(s_j))\right] $$ 
\[
 \E\left[\prod_{i=0}^{m_1} f_i((X_{\text{crp}}^{\pi_n}(s_i)) \prod_{j=0}^{m_2}g_j( Z_{\text{crp}}^{\pi_n}(t_j))
  )  \right] \to \E \left[  \prod_{i=0}^{m_1} f_i((X_{\text{sc}}(s_i)) \prod_{j=0}^{m_2}g_j( Z_{\text{sc}}(t_j)))\right], \ \text{as } n \to \infty.
  \]
\label{Cor} 
\end{cor}

Of course, we hope that Proposition \ref{Theorem} together with Corollary \ref{Cor} imply the
weak convergence of $(X_{\text{crp}}^{\pi_n}, Z_{\text{crp}}^{\pi_n})$ to $(X_{\text{sc}}, Z_{\text{sc}})$ at the process level.  Since it is natural to view $(X_{\text{crp}}^{\pi_n}, Z_{\text{crp}}^{\pi_n}) \in \R^{2d}$, we would ideally like to show that $(X_{\text{crp}}^{\pi_n}, Z_{\text{crp}}^{\pi_n}) \implies (X_{\text{sc}}, Z_{\text{sc}})$ weakly as stochastic processes on $\R^{2d}$.   For  such convergence to hold we require the laws of 
$\{ (X_{\text{crp}}^{\pi_n}, Z_{\text{crp}}^{\pi_n})\}$ to be relatively compact (i.e. every sequence has a convergent subsequence) %mpact with respect to Prohorov's metric corresponding to the Skorohod metric on $D_{\R^{2d}}[0,\infty)$).  
Unfortunately, and perhaps surprisingly, this is not the case as we now show.

The following  result is Theorem 7.2 on page 128 of  \cite{Kurtz86}.  %It provides  necessary and sufficient conditions for the relative compactness of a set of measures  in the Prohorov metric induced by the Skorohod metric. 
Following the notation in \cite{Kurtz86}, when $E$ is a metric space we let $D_E[0, \infty)$ be the set of all c\`adl\`ag functions from  $[0,\infty)$ to $E$.
\begin{thm}
Let $(E,r)$ be a complete and separable metric space, and let $\{X_n\}$ be a family of processes with sample paths in $D_E[0,\infty)$ endowed with the Skorohod metric. Then $\{X_n\}$ is relatively compact if and only if the
following two conditions hold: 
\begin{enumerate}[1.]
\item For each $\eta>0$ and rational $t\ge 0$, there is exists a compact set $\Gamma_{\eta,t} \subset E$ such that  
$$\inf_n P (X_n(t) \in \Gamma_{\eta, t})  \geq 1- \eta. $$ 
\item For every $\eta >0$ and $T>0$, there exists $\delta > 0$ such that 
$$\sup_{n} P(w'(X_n, \delta, T) \geq \eta) < \eta$$
where   
$$w'(X, \delta, T) \eqdef  \inf_{\pi} \max_i \sup_{a, b \in [t_i, t_{i+1})}
|X(a)- X(b)|   $$
where $\pi$ ranges over all partitions of $[0,T]$ satisfying $t_{i+1} - t_{i}
> \delta$ for all $i\ge 0$.   
\end{enumerate}     \label{Precompact}
\end{thm}

Unfortunately the conditions of Theorem \ref{Precompact} do not
hold in general for our set of processes  $\{ (X_{\text{crp}}^{\pi_n}, Z_{\text{crp}}^{\pi_n})\} $ over the skorohod space $D_{\RR^{2d}} [ 0, \infty)$.
To see this,  we note the following two facts:  
\begin{enumerate}[1.]
\item For jump processes whose jump sizes are bounded below, for example by integer values in our present setting,  for small enough $\eta>0$ we have
\[	
	\{w'((X,Z),\delta,T)<\eta\} = \{w'((X,Z), \delta, T)= 0\},
\]
\item  The event $ w'((X,Z), \delta, T)= 0$  can be achieved if and  only if 
 the minimum time between jumps of $(X, Z)$ is greater than $\delta$.  
 \end{enumerate}
 To understand the second statement, simply note that if the minimum time between jumps is less than $\delta$, then for any partition $\pi$ satisfying $t_{i+1}-t_i>\delta$ for all $i$, the process must change by at least the smallest jump size ($\min_k |\zeta_k|$ in our case) in \textit{some} interval of the partition. Conversely, if the minimum holding time of the process is  greater than $\delta$, then we achieve a value of $0$ for $w'$ by choosing $\pi$ so that the jump times correspond with a subset of the partition times $t_i$. 
 % have a partitionSupposes that $|\zeta_k| > C$ for all $k$. Then for any partition with $\Delta t_i > \delta$,  $w((X,Z), \delta, T) \geq C$ for all realization of $(X,Z)$ for which the inter-jump interval is smaller than $\delta$. \\       
% \textcolor{red}{ NOTE :::: For compactness, \ref{Precompact} is necessary, and we cannot avoid his. }  

The following example explicitly shows that Theorem \ref{Precompact} does not  hold for our choice of $\{ (X_{\text{crp}}^{\pi_n}, Z_{\text{crp}}^{\pi_n})\} $ with $E = \RR^{2d}$.  Essentially the same argument would work for any model considered in this paper.

 \begin{example}\label{ex:408957}
   Consider the  chemical reaction network
 \begin{align*}
A \to 2A,% 0 &{\to} A \tag{1} \label{a1}
% 0  &{\to} A \tag{2} \label{a2} 
\end{align*}
which models increases in $A$ as a counting process with a linear intensity (i.e. a linear birth process).  We consider the corresponding coupled processes $(X_{\text{sc}},Z_{\text{sc}})$ and $(X_{\text{crp}}^{\pi_n},Z_{\text{crp}}^{\pi_n})$ with 
\[
	\lambda_1(x) = \theta x, ~~~~~ \widetilde \lambda_1(x) = (\theta + h)x,
\]
and initial condition 
\[
	X_{\text{sc}}(0) = Z_{\text{sc}}(0) =X_{\text{crp}}^{\pi_n}(0) = Z_{\text{crp}}^{\pi_n}(0) >0.
\]
For any $\delta>0$, the probability that the processes $X_{\text{sc}}$ and $Z_{\text{sc}}$  jump simultaneously in the time period $[0,\delta]$ \textit{and} that their simultaneous jump is the first jump for both processes is
\[
	\alpha_\delta \eqdef \frac{\theta}{\theta+h} \left( 1 - e^{-(\theta + h)X(0) \delta}\right)>0.
\]
 By the arguments we made in the proof above, for any $\epsilon >0$  there exists some $M_\epsilon$
such that if  $n > M_\epsilon$, then with probability greater than $\alpha_\delta -
\epsilon$,  both $X_{\text{crp}}^{\pi_n}$ and $Z_{\text{crp}}^{\pi_n}$ will also make
a first jump in $[0,\delta]$. However, with a probability of one, $X_{\text{crp}}^{\pi_n}$ and $Z_{\text{crp}}^{\pi_n}$   jump at different times.  
  Hence, when they jump in the time interval $[0,\delta)$,  we have 
\begin{align*}
	\sup_{a,b\in [0,\delta)} | (X^{\pi_n}_{\text{crp}},Z^{\pi_n}_{\text{crp}})(a)  -    (X^{\pi_n}_{\text{crp}},Z^{\pi_n}_{\text{crp}})(b)| \ge 1.
\end{align*}
 This in particular means that for
any $0< \eta < 1$, 
 \[
 	\sup _n P (w'( (X_{\text{crp}}^{\pi_n},Z_{\text{crp}}^{\pi_n}) , \delta, T) \geq \eta)
 \geq \alpha_\delta,
 \]
 and the laws of $\{ (X_{\text{crp}}^{\pi_n}, Z_{\text{crp}}^{\pi_n})\}$ fail to be relatively compact.  %Since relative compactness of the laws of the processes holds if and only if every sequence has a convergent subsequence, we may conclude that $\{ (X_{\text{crp}}^{\pi_n}, Z_{\text{crp}}^{\pi_n})\}$ does not converge weakly in $D_{\R^2}[0,\infty)$.
\end{example}

\subsection{Weak Convergence in the product Skorohod topology}
\label{sec:weak_conv}

Example \ref{ex:408957} demonstrates that the measures induced by $(X_{\text{crp}}^{\pi_n}, Z_{\text{crp}}^{\pi_n})$ on $D_{\R^{2d}}[0,\infty)$ are not relatively compact.  Hence, the processes $(X_{\text{crp}}^{\pi_n}, Z_{\text{crp}}^{\pi_n})$ do not converge weakly to $(X_\text{sc},Z_{\text{sc}})$ in  $D_{\R^{2d}}[0,\infty)$.  However, in this section we demonstrate that there is convergence  in
$$\cD :=D_{\RR^d}[0, \infty) \times D_{\RR^d}[0, \infty)$$
endowed with the product Skorohod topology.

As is usual, the main work that remains to be done is in showing that $\{ (X_{\text{crp}}^{\pi_n}, Z_{\text{crp}}^{\pi_n})\}$ is relatively compact in the appropriate topological space.  

\begin{prop}\label{prop:30498}
Let $\cD   \eqdef D_{\RR^d}[0, \infty) \times D_{\RR^d}[0, \infty)$, with the product Skorohod topology.  The family of processes   $\{ (X_{\text{crp}}^{\pi_n}, Z_{\text{crp}}^{\pi_n})\}$ is relatively compact in  $\cD$.
\end{prop}
\begin{proof}
%We would like to show that $(X_{\text{crp}}^{\pi_n}, Z_{\text{crp}}^{\pi_n})$ satisfies the conditions in Theorem \ref{Precompact} with respect to the space $\cD$. 

By Theorem 2.2 on page 104 of \cite{Kurtz86}, it is enough to show that 
% show that $(X_{\text{crp}}^{\pi_n}, Z_{\text{crp}}^{\pi_n})$ is relatively compact we need to
%show that, 
for any $\epsilon > 0$,  
 there exists a compact set $C^\epsilon\in \cD$ such that 
$$ \inf_n P ((X_{\text{crp}}^{\pi_n}, Z_{\text{crp}}^{\pi_n}) \in C^\epsilon) > 1- \epsilon.$$ 
To show this, we consider the marginal processes, which we recall
%Recall that the marginals of $(X_{\text{crp}}^{\pi_n}, Z_{\text{crp}}^{\pi_n})$ 
satisfy  $X \sim X_{\text{crp}}^{\pi_n}$ and  $  Z\sim Z_{\text{crp}}^{\pi_n}$ for each $n \ge 1$. Note that if $A^\epsilon, B^\epsilon \subset D_{\RR^d}[0,\infty)$ are compact, then the inequalities
\begin{align}
\begin{split}
	P (X \in A^\epsilon) &= P (X_{\text{crp}}^{\pi_n} \in A^\epsilon) > 1- \frac{\epsilon}2\\
 	 P (Z \in B^\epsilon)&= P (Z_{\text{crp}}^{\pi_n} \in B^\epsilon) > 1- \frac{\epsilon}2
 \end{split}
	 \label{eq:490875}
\end{align}
imply the inequality
\[
P ((X_{\text{crp}}^{\pi_n}, Z_{\text{crp}}^{\pi_n}) \in A^\epsilon\times B^\epsilon) > 1- \epsilon,
\]
with $A^\epsilon\times B^\epsilon$ compact in $\cD$.
Hence, it is sufficient to simply prove the pair of  inequalities \eqref{eq:490875} for the marginal processes, which  live in $D_{\R^d}[0,\infty)$.  However, inequality \eqref{eq:490875} holds so long as the marginal processes are tight (in $D_{\R^d}[0,\infty)$), and so Theorem \ref{Precompact} may be used.
 Therefore it
suffices to show that $X$ and $Z$ both separately satisfy the conditions in
Theorem \ref{Precompact}, which we do now.

Since $X$ is a nonexplosive pure jump process, it clearly
passes the first condition of Theorem \ref{Precompact}. Also, recall that $X$ is constructed with $R\in \Z_{>0}$ Poisson processes, one for each jump direction.  Then for any $T >0$ and $M>0$, 
\begin{align}
\begin{split}
P(w'(X, \delta, T) > 0  )   &\leq P\left(w'(X, \delta, T) > 0, \sup_{k=1,..,R, s <
T} \lambda_k(X(s)) \leq M\right)\\
&\hspace{.3in} +   P\left( \sup_{k=1,..,R, s <T} \lambda_k(X(s)) > M \right) \\
&\le P \left(w'(Y(MR \cdot), \delta, T) > 0\right) + P\left( \sup_{k=1,..,R, s <T}
\lambda_k(X(s)) > M\right)
\end{split}
\end{align}
where  $Y(M R\cdot)$ is a Poisson process with rate $MR$. Since $X$ is non explosive, we
may take $M$ large enough to control the second piece, and for this
$M$ we can choose $\delta$ small enough to control the first
piece. That is,
$\lim_{\delta \to 0 }P( w'(X, \delta, T) > 0 ) = 0$.
This tells us that $X$ also passes the second condition of Theorem \ref{Precompact}. The same procedure works for $Z$. Thus,  $\{ (X_{\text{crp}}^{\pi_n}, Z_{\text{crp}}^{\pi_n}) \} $ is relatively compact in $\cD$ with the product
topology.  
\end{proof}

With this proposition at our hand, we can prove the main result of our paper.

\begin{thm}
Suppose $X$ and $Z$ are both non-explosive, c\`adl\`ag process as given above.
Let $D_{\RR^d}[0, \infty)$ be the Skorohod Space as defined in \cite{Kurtz86}. Consider  the
product topology on 
$$\cD :=D_{\RR^d}[0, \infty) \times D_{\RR^d}[0, \infty).$$
Also, let $\pi_n = \{s_j^n\} $ be a sequence of partitions of
$[0,\infty)$ such that  
\[
 	\text{mesh}(\pi_n) = \max_{j < \infty} (s_{j}^n - s_{j-1}^n) \to 0, \quad \text{as } n \to \infty.
\]  
Then for all $f : \cD \to \RR$ that are bounded and
continuous,
\[
	\E[f(X_{\text{crp}}^{\pi_n}, Z_{\text{crp}}^{\pi_n})] \to \E[f(X_{\text{sc}}, Z_{\text{sc}})], \ \ \ \text{as } n \to \infty.
\]
That is,  $(X_{\text{crp}}^{\pi_n}, Z_{\text{crp}}^{\pi_n} ) \to  (X_{\text{sc}}, Z_{\text{sc}})$, as $n \to \infty$, weakly in the product Skorohod topology. 
\label{Goal}
\end{thm}

We would like to emphasize that  the test function $f$ considered above maps a path in  $\cD$ to $\RR$.   The test functions for Proposition  \ref{Theorem}, on the other hand, are evaluated at discrete time points.

Now we put everything together to prove Theorem
\ref{Goal}.
\begin{proof}[Proof of Theorem \ref{Goal}] 

%Proposition \ref{Theorem} gives convergence of the finite dimensional distributions and  Proposition \ref{prop:30498} gives relative compactness.  Hence, the proof of Theorem \ref{Goal} is standard.  We include it for completeness.

By Proposition \ref{prop:30498} it is sufficient to show that every convergent (in distribution) subsequence of $(X_{\text{crp}}^{\pi_n}, Z_{\text{crp}}^{\pi_n} )$ converges in distribution to $(X_{\text{sc}}, Z_{\text{sc}})$.
By Corollary \ref{Cor}, it is sufficient to show that if
%all we need to show is that  when two sub-sequential limits of the sequence of distributions of $\{(X_{\text{crp}}^{\pi_n}, Z_{\text{crp}}^{\pi_n} )\}$ have the same 
%marginals at every set of finite time points, then the two distributions are identical. More precisely, we must show that if 
\begin{align}
	\E \left[  \prod_{i=0}^{m_1} f_i(X_{\text{sc}}(s_i)) \prod_{j=0}^{m_2}g_j( Z_{\text{sc}}(t_j))
     \right]  &=  \E \left[  \prod_{i=0}^{m_1} f_i(X^*(s_i)) \prod_{j=0}^{m_2}g_j( Z^*(t_j)) 
     \right]
      \label{marginal}
  \end{align}
   for all $\{ s_i\}, \{ t_j\} \subset [0,\infty),$ and $f_i, g_i \in \overline C(\RR^d)$ (bounded and continuous functions), 
then $\E[h(X_{sc}, Z_{sc})]= \E[h(X^*, Z^*)]$ for any bounded and continuous function $h: \cD \to \R$.  A standard monotone  class argument (for example, see page 132 in \cite{Kurtz86}) shows that \ref{marginal} is more than enough to guarantee that $\E[h(X_{\text{sc}}, Z_{\text{sc}})] = \E[h(X^*, Z^*)]$ for all 
 $h$ continuous with respect to $D_{\RR^{2d}}[0, \infty).$  From the definition of the Skorohod metric, it is straightforward to show that the  topology of $D_{\RR^{2d}}[0, \infty)$ is finer than that of $\cD$. This in particular means that the continuous functions with respect to $\cD$ are a subset of those of $D_{\RR^{2d}}[0, \infty)$.  Thus,  we may conclude that $\E[h(X_{\text{sc}}, Z_{\text{sc}})] = \E[h(X^*, Z^*)]$ if $h$ is continuous with respect to $\cD$, and the result is shown.
%To show \eqref{marginal}, we first note that if $f$ is continuous with respect to $\cD$, it is also continuous with respect to $D_{\RR^{2d}}[0, \infty)$ (the topology of $D_{\RR^{2d}}[0, \infty)$ is finer than that of $\cD$). 
%%% They have same Borel Algebra, but the latter is coarser (Billingsley) 
%A standard monotone  class argument (for example, see page 132 in \cite{Kurtz86}) shows that \ref{marginal} is more than enough to guarantee that $\E[h(X_{\text{sc}}, Z_{\text{sc}})] = \E[h(X^*, Z^*)]$ for all 
% $h$ continuous with respect to $D_{\RR^{2d}[0, \infty)}.$  In particular, the statement holds for all $h = f \otimes g$, where both  $f$ and $g$ are respectively continuous with respect to 
% $D_{\RR^{d}[0, \infty)}.$         
\end{proof}

While the results presented so far pertain to the specific couplings found in the numerical analysis literature, a slightly more general theorem can be achieved by following an identical line of reasoning.  
\begin{thm}\label{general} 
For $i \in \{1,2,3\}$ and $k \in \{1,\dots,R\}$, let $r_{ik}:\RR^d\times \RR^d \to \RR_{\ge 0}$  be a non-negative measurable function.  Suppose that $\{\pi_n\}$ is a sequence of partitions  of $[0,\infty)$ for which $\text{mesh}(\pi_n) \to 0$, as $n \to \infty$.  Define $(X_{\text{sc}},Z_{\text{sc}})$ and   $(X^{\pi_n}_{\text{crp}},Z^{\pi_n}_{\text{crp}})$ via
\begin{align*}
X_{\text{sc}}(t) = X(0) + &\sum_{k=1}^R\Bigg\{ Y_{1k} \left(
 \int_0^t r_{1k}(X_{\text{sc}}, Z_{\text{sc}})(s)  ds \right)  +   Y_{2k}
\left( \int_0^t r_{2k}(X_{\text{sc}}, Z_{\text{sc}})(s)  ds \right)
\Bigg\}  \zeta_k \\
Z_{\text{sc}}(t) = Z(0) + &\sum_{k=1}^R \Bigg\{ Y_{1k} \left( \int_0^t
  r_{1k}(X_{\text{sc}}, Z_{\text{sc}})(s)  ds
\right) 
 +  Y_{3k} \left( \int_0^t   r_{3k}(X_{\text{sc}}, Z_{\text{sc}})(s) ds  \right)  \Bigg\} \zeta_k,
\end{align*}
and 
\begin{align*}
 X_{\text{crp}}^{\pi_n}(t)  &= X(0) + \sum_{m=0}^{\infty} \sum_{k =1}^R Y_{km}^n \left(
  \int_{t \wedge  s_m}^{t \wedge s_{m+1}} \{ r_{1k}(X_{\text{crp}}^{\pi_n}, Z_{\text{crp}}^{\pi_n})(s) +  r_{2k}(X_{\text{crp}}^{\pi_n}, Z_{\text{crp}}^{\pi_n})(s) \} ds  \right)
\zeta_k \\
 Z_{\text{crp}}^{\pi_n} (t) &= Z(0) + \sum_{m=0}^{\infty}  \sum_{k =1}^R Y_{km}^n \left(
   \int_{t \wedge  s_m}^{t \wedge s_{m+1}}  \{r_{1k}(X_{\text{crp}}^{\pi_n}, Z_{\text{crp}}^{\pi_n})(s) +  r_{3k}(X_{\text{crp}}^{\pi_n}, Z_{\text{crp}}^{\pi_n})(s)\} ds  \right)
\zeta_k,
\end{align*}
where all notation is as before.
Finally, we suppose that all processes are non-explosive.  Then, $(X_{\text{crp}}^{\pi_n}, Z_{\text{crp}}^{\pi_n} ) \to  (X_{\text{sc}}, Z_{\text{sc}})$, as $n \to \infty$, weakly in the product Skorohod topology. 
\end{thm} 

Proposition \ref{Goal} is therefore a special case of Theorem \ref{general} in which each $r_{ik}$ depends on $\lambda_k$ and $\tilde \lambda_k$ in a  specific way.    

\section{Numerical examples}
\label{sec:examples}
In this section, we provide two numerical examples demonstrating the convergence of the local-CRP coupling to that of the split coupling.  Based upon our motivation in terms of variance reduction, we  focus upon the convergence of the variance between the coupled processes.

\begin{example}\label{example1}
We begin by considering a basic model of gene transcription and translation, where the model tracks the counts for the numbers of genes ($G$), mRNA molecules ($M$), and proteins ($P$) in the system.   We suppose that the system can undergo the following possible reactions, 
 \begin{align*}
 	G  &\rightarrow G+M\tag{R1}\\
	M  &\rightarrow   M+P\tag{R2}\\
	M &\rightarrow \emptyset\tag{R3}\\
	P  &\rightarrow  \emptyset,\tag{R4}
\end{align*}
where, for example, reaction (R1) implies a net change to the system of one extra mRNA molecule.  Since no reaction changes the number of genes present in the system, we may take that to be a fixed quantity.  Hence, there are two dynamic components, and the stochastic model for this system is
\begin{align*}
	X(t) = X(0) &+ Y_1\left( \int_0^t \lambda_1(X(s)) ds\right) \left[\begin{array}{c} 1 \\ 0\end{array}\right] +  Y_2\left( \int_0^t \lambda_2(X(s)) ds\right) \left[\begin{array}{c} 0 \\ 1\end{array}\right]\\
	&+ Y_3\left( \int_0^t \lambda_3(X(s)) ds\right) \left[\begin{array}{c} -1 \\ 0\end{array}\right]+ Y_4\left( \int_0^t \lambda_4(X(s)) ds\right) \left[\begin{array}{c} 0\\-1\end{array}\right],
\end{align*}
where $X_1$ counts the numbers of mRNA molecules, and $X_2$ counts the numbers of proteins.  We now let $X$ be the process with intensity functions 
\begin{align*}
	\lambda_1(x) = 2, \quad \lambda_2(x) = 10x_1, \quad \lambda_3(x) = (1/4 + 1/80)x_1, \quad \lambda_4(x) = x_2,
\end{align*}
and let $Z$ be the process with intensity functions
\begin{align*}
	\lambda_1(x) = 2, \quad \lambda_2(x) = 10x_1, \quad \lambda_3(x) = (1/4 - 1/80)x_1, \quad \lambda_4(x) = x_2.
\end{align*}
These are reasonable choices, for example, if we were attempting to estimate the sensitivity of some statistic with respect to the rate parameter for the third intensity function evaluated at $1/4$.

Let $\pi_n$ be a partition of $[0,30]$ into $n$ equally sized intervals.  In Figure \ref{fig:figure1}, we plot  numerical estimates of $Var(X_{\text{sc}}(t)  - Z_{\text{sc}}(t))$, $Var(X_{\text{crp}}(t)  - Z_{\text{crp}}(t)),$ and $Var(X^{\pi_n}_{\text{crp}}(t)  - Z^{\pi_n}_{\text{crp}}(t))$, for $n \in \{2,6,30,300\}$, over the time period $[0,30]$.  The estimates were achieved via Monte Carlo methods with 10,000 sample paths.
  \begin{figure}[t]
  	\begin{center}
	\includegraphics[scale = 0.6]{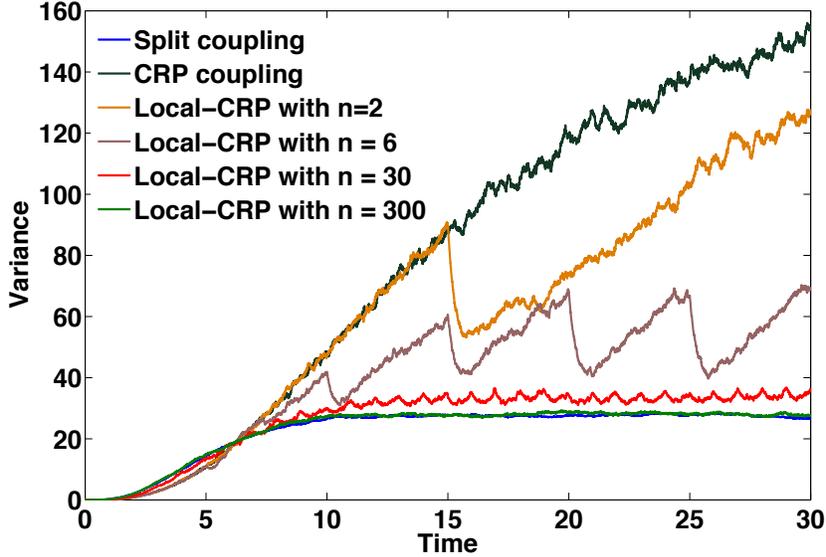}
	\end{center}
	\caption{Numerical approximations (via Monte Carlo with 10,000 sample paths) for the variance of the difference between the processes $X$ and $Z$ of Example \ref{example1} for the split coupling (blue), CRP coupling (black), and various local-CRP couplings.   Convergence of the variance of the local-CRP coupling to the variance of the split coupling is clear.} 
	\label{fig:figure1}
\end{figure}
We observe the uniform convergence of $Var(X_{\text{crp}}^{\pi_n}(\cdot)  - Z_{\text{crp}}^{\pi_n}(\cdot))$ to  $Var(X_{sc}(\cdot)  - Z_{sc}(\cdot))$ as $\text{mesh}(\pi_n) \to 0$.    We also observe a sharp drop in the variance of $X^{\pi_n}_{\text{crp}}(\cdot)  - Z_{\text{crp}}^{\pi_n}(\cdot)$  at the ``resetting'' of the Poisson processes, which occur at the end of each interval of the discretization $\pi_n$. % This supports the argument that the "resetting" has the ability to bring $X$ and $Z$ together, or the ability to recouple the processes in this specific example.  This was phenomena was observed for all studied examples as well, including the quadratic birth and death example that we will present momentarily.      Quantification of this phenomena is an obvious venue for the future research. 
	
\end{example}

\begin{example}\label{example2}
	Consider a simple quadratic birth and death model  
	\begin{align*}
		\emptyset &\rightarrow 2A\tag{r1}\\
		2A &\rightarrow \emptyset\tag{r2}
	\end{align*} 
	with initial count $X(0)$ given by a Poisson random variable with parameter $15$.   
We can model the dynamics of this system with the stochastic equations
\[
	X(t) = X(0) + 2Y_1\left(\int_0^t \lambda_1(X(s))ds\right) - 2Y_2\left(\int_0^t \lambda_2(X(s))ds\right), 
\]
where
\[
	\lambda_1(x) = 400,  \quad \text{and} \quad \lambda_2(x) = kx(x -1),
\]
and where $k$ is a parameter of the model.  We consider the model $X$ with $k = 0.1 + 1/25$ and the model $Z$ with $k = 0.1 - 1/25$.  Further, we let the initial conditions of $X$ and $Z$ be independent Poisson random variables with a parameter of 15 (that is, the initial conditions of $X$ and $Z$ are independent from each other).  Let $\pi_n$ be a partition of $[0,1]$ into $n$ equally sized intervals.  In Figure \ref{figure2}, we plot  numerical estimates of $Var(X_{\text{sc}}(t)  - Z_{\text{sc}}(t))$, $Var(X_{\text{crp}}(t)  - Z_{\text{crp}}(t)),$ and $Var(X^{\pi_n}_{\text{crp}}(t)  - Z^{\pi_n}_{\text{crp}}(t))$, for $n \in \{2,4,8,100\}$, over the time period $[0,1]$.  The estimates were achieved via Monte Carlo methods with 5,000 sample paths. We again observe the sharp drop in variance at the ``resetting'' times of the processes.  
\begin{figure}[t]
	\begin{center}
	\includegraphics[scale = 0.6]{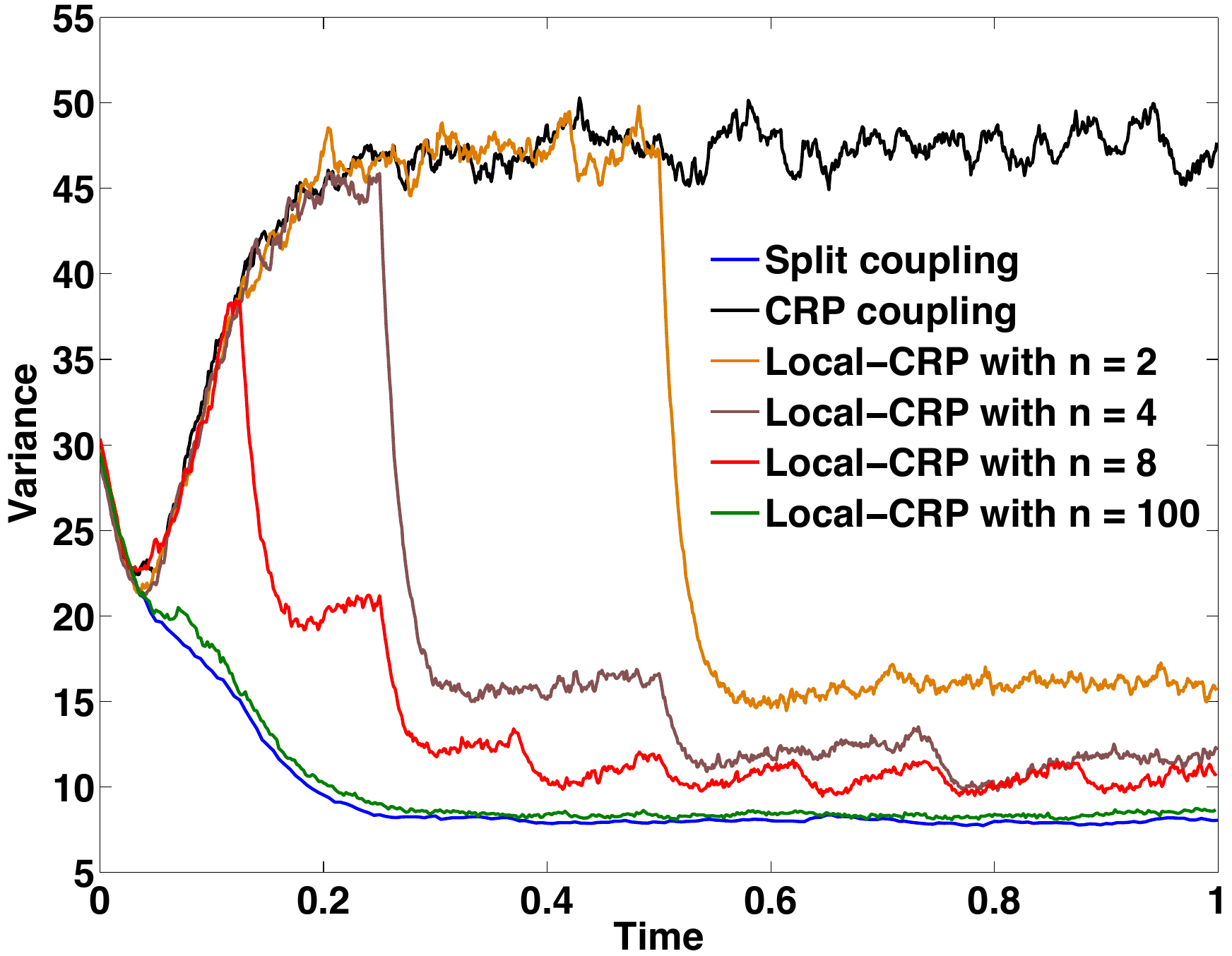}
		\caption{Numerical approximations (via Monte Carlo with 5,000 sample paths) for the variance of the difference between the processes $X$ and $Z$ of Example \ref{example2} for the split coupling (blue), CRP coupling (black), and various local-CRP couplings.   Convergence of the variance of the local-CRP coupling to the variance of the split coupling is clear.} 
	\end{center}	
	\end{figure}
	\label{figure2}
\end{example}

%\section{Conclusion}
%\label{sec:conclusion}
%As explained in the introductory section, there is still much to be
%done for the research of coupling the processes in \eqref{system0}.  While Kurtz' coupling performs
%well in many situations,  there has never been any result on
%particular statistics that it optimizes. This study, however,
%revealed at least one source of the advantage of Kurtz' coupling. 
%Recall that the local-CRP was created in hope to \eqref{localCRP}
%``reset''  the Poisson channels at every small intervals to prevent the fast decoupling. 
% The study shows that Kurtz' coupling is, in essense, continuously
% ``resetting''  of the rate of each Poisson channels. Intuitively speaking,  Kurtz' coupling is constructed in such a way
%that at every infinitesimal time interval, as much effort  as possible
%is done to make $X$ and $Z$ jump
%simultaneously.  Also, the
% performance of couplings of \eqref{system0} has been tested for MLMC and for
% sensitivity analysis only for the purpose of taking
% differences. While Kurtz' coupling does well in reducing the variance
% of  $$g(X,Z) = X-Z,$$ 
%it might not do so well for other functions of $g$. For all purposes,
%hall investigate and explore the systematic ways construct the
%couplings to reduce the variance for many different types of $g$.  
%

\section{Discussion}
\label{sec:discussion}

 The stochastic models finding widespread use in the cell biology literature are typically immensely complicated, and computational methods often provide the only effective way to probe the dynamics.  As Persi Diaconis recently noted \cite{Diaconis2013}, this presents  mathematicians with an opportunity to make contributions by explicitly studying the different simulation and computational algorithms themselves.  Such analyses will not only shed light on which methods to use in different contexts, but will inevitably lead to a deeper understanding of the underlying processes, and hence to better computational methods.

In this work we have clarified the connection between two couplings commonly found in the computational cell biology literature and, in particular, showed that the split coupling can be regarded as a natural limit of a localized version of the CRP coupling.  
%The proof illuminates that the split coupling is in a true sense an infinitesimal coupling for the two Poisson driven point processes.  DAVE: I don't really know what this sentence means.  Let's leave it out so we can remain precise.
There are other interesting ways to understand the split coupling.  For example, Arampatzis and Katsoulakis \cite{Markos} recently studied a group of couplings that is included in the family of  general split couplings considered in Theorem \ref{general}. They note that for each  test function $f$
there is an optimal choice for the function  
$r_{1k}(\lambda_k, \widetilde \lambda_k, \mathcal U, \mathcal V)(s)$ in \eqref{couple_rates}  
that minimizes the variance of the finite difference $\E[f(X_t) - f(Z_t)]$ in the setting of \eqref{eq:split_coupling}. 
When the test function is  $f(x) = x$, the correct choice of $r_{1k}$ is the one given in \eqref{couple_rates}, which yields the split coupling.   %We would like to note that there is an analogue of the local-CRP coupling that will converge to this 
%``general split coupling"  in the same topology that we discussed here.   
%Together with the fact that there is sometimes a domain (the case of small terminal time, for example)  over which the CRP coupling performs better than the split coupling,  in that it provides a lower variance, the above implies that there may be a possible improvement to the CRP coupling in cases of importance.  %Also, while we revealed the meaning of the infinitesimal coupling of two processes whose rates share a same additive component (e.g. $r_1 + r_2$  and $r_1 + r_3$) , we do not know yet the analogue of the infinitesimal coupling for the processes whose rates share a multiplicative component (e.g. $r_1 r_2$ and $r_1 r_3$.) This will be a venue of future research as well.  \\   \indent  Multiplicative rates are used widely in the models for neural spikes, for example, in Koerding et al \cite{Stevenson}.  In a typical model of neural spike analysis, each $Y_k$ used this paper represents a counting process for the spikes of the $k$th neuron. Hence the correlations among the random time change of $Y_k$s play an essential role.  Advances in the understanding of the couplings for these processes will potentially benefit various biological applications.  
%DAVE: I don't think this needs to be mentioned here.  

\vspace{.2in}
\noindent\textbf{Acknowledgments.}
	  We thank Thomas Kurtz for several illuminating discussions and for reading an early version of this work.  We thank two anonymous referees for careful readings that substantially improved the clarity of Section \ref{sec:weak_conv}.  Anderson was supported by NSF grants DMS-1009275 and DMS-1318832 and Army Research Office grant W911NF-14-1-0401.  Koyama was supported by NSF grants DMS-1009275, DMS-1318832, and DMS-0805793.
\bibliographystyle{amsplain} 
\bibliography{CRP_CFD.bib}

\end{document}